\documentclass[12pt]{amsart}

\usepackage{amssymb}
\usepackage{amsthm}
\usepackage{amsmath}
\usepackage{amscd}
\usepackage{verbatim}
\usepackage[all]{xy}

\numberwithin{equation}{section}

\theoremstyle{plain}
\newtheorem{theorem}{Theorem}[section]
\newtheorem{corollary}[theorem]{Corollary}
\newtheorem{lemma}[theorem]{Lemma}
\newtheorem{proposition}[theorem]{Proposition}

\theoremstyle{definition}
\newtheorem{definition}[theorem]{Definition}
\newtheorem{remark}[theorem]{Remark}
\newtheorem{example}[theorem]{Example}

\theoremstyle{remark}

\newcommand{\OO}{\mathcal O}
\newcommand{\A}{\mathbb{A}}
\newcommand{\R}{\mathbb{R}}

\newcommand{\Q}{\mathbb{Q}}
\newcommand{\Z}{\mathbb{Z}}
\newcommand{\N}{\mathbb{N}}

\renewcommand{\H}{\mathbb{H}}

\newcommand{\M}{\mathbb{M}}

\newcommand{\kzxz}[4]{\left(\begin{smallmatrix} #1 & #2 \\ #3 & #4\end{smallmatrix}\right) }

\newcommand{\calP}{\mathcal{P}}

\newcommand{\vol}{\operatorname{vol}}
\newcommand{\tr}{\operatorname{tr}}

\newcommand{\Sp}{\operatorname{Sp}}

\newcommand{\Hom}{\operatorname{Hom}}
\newcommand{\Aut}{\operatorname{Aut}}

\newcommand{\GL}{\operatorname{GL}}
\newcommand{\SO}{\operatorname{SO}}

\newcommand{\cha}{\operatorname{char}}
\newcommand{\Ind}{\operatorname{Ind}}

\newcommand{\ord}{\operatorname{ord}}

\newcommand{\ff}{\hbox{if }}
\newcommand{\SL}{\operatorname{SL}}
\newcommand{\gen}{\operatorname{gen}}

\newcommand{\rank}{\operatorname{rank}}

\begin{document}

\title{Weak Siegel-Weil formula for $\M_{2}(\mathbb{Q})$ and arithmetic on quaternions}

\date{\today}
\author[Tuoping Du]{Tuoping Du}
\address{Department of Mathematics, Northwest University, Xi'an, 710127 ,  P.R. China}
\email{tpdu@nwu.edu.cn}
\subjclass[2010]{11R52, 11G18, 11F32, 11F41, 11S23}


\begin{abstract}
We prove weak Siegel-Weil formula for the space $\M_{2}(\mathbb{Q})$. By this formula and the Siegel-Weil formula, we give the explicit formulae for Hecke correspondence's degree and average representation numbers over genus associated to Eichler orders. 
At last, we give the explicit formulae for representations of a number as sums of three squares or four squares by local Whittaker functions, and find that these functions are exactly the local factors of Hardy's singular series.

\end{abstract}

\thanks{thanks to}


\maketitle

\section{Introduction}
Weil gave the relation between the value of  Eisenstein series and integral of theta function in the paper \cite{Weil2}, which is called the Siegel-Weil formula.
 It plays a very important role in number theory and arithmetic geometry. 
 
 In this paper, 
 we mainly study the arithmetic and geometry on quaternion algebras.  By the Siegel-Weil formula, we give the explicit formulae for Hecke correspondence's degree and average representation numbers over genus. We identify these numbers with Fourier coefficients of Eisenstein series, which could be written as infinite products of local Whittaker functions. 
 
There are mainly two ways to study representation numbers of positive definite quadratic forms, that is, the Siegel-Weil formula and the cycle method.
 By the Siegel-Weil formula, we give the exactly formulae for representation numbers of three squares and four squares sums.  Hardy \cite{Ha1}\cite{Ha2} studied the representation number via singular series, which is an infinite product. We find these two methods locally are the same, i.e., the local factors of singular series are equal to local Whittaker functions.

Siegel-Weil formula holds for all quaternion algebras over $\Q$ except the space $\M_{2}(\mathbb{Q})$.  For this space, we 
 prove the Siegel-Weil formula except the constant term  and call it weak Siegel-Weil formula (Theorem \ref{weakformula}) in this paper.

Let us recall the classical Siegel-Weil formula of the orthogonal type.
 Let $(V, Q)$ be a quadratic space over $\Q $ with even dimension $m$. For the reductive dual  pair $G=\Sp_{n}$ and $H=O(V)$, one has the Weil representation $\omega$ of group $G(\A) \times H(\A)$ which acts on $S(V^{n}(\mathbb{A}))$.
For convenience, we assume $n=1$, hence $G=\SL_{2}$.

For an algebraic group $W$ over $\Q$, set $[W]=W(\Q) \backslash W(\A)$.
Then the theta kernel (\cite{Weil})
\begin{equation}
\theta(g, h, \varphi)=\sum_{x\in V(\mathbb{Q})} \omega(g)\varphi(h^{-1}x),
\end{equation}
is an automorphic form on $[G]\times [H]$, where $ \varphi \in S(V(\A)), g \in G(\A), h\in H(\A)$.   So
the theta  integral
\begin{equation}
I(g, \varphi)=\int_{[H]} \theta(g, h, \varphi)dh
\end{equation}
is an automorphic form on $[G]$ if the integral is absolutely convergent.

There is another way to construct automorphic forms from $\varphi \in S(V(\A))$. For $s \in \mathbb{C}$, let
 $I(s, \chi_V)= \Ind_{P}^{G}( | | ^{s}\chi_V)$ be the induced representation
of $G(\A)$ consists of smooth functions $\Phi(g, s)$ on $G(\A)$.
The Eisenstein series is defined by
\begin{equation}
E(g, s, \Phi)= \sum_{\gamma\in P \setminus G}
\Phi(\gamma g, s),
\end{equation}
where $P$ is the parabolic subgroup of $G$. 
There is a $\SL_2(\A)$-intertwining map ($s_0= \frac{m}2 -1$)
\begin{equation}
\lambda=\lambda_{V} : S(V(\mathbb{A})) \rightarrow I(s_{0}, \chi_{V}), \quad \lambda(\varphi)(g)= \omega(g)(0).
\end{equation}
It is defined locally, and we write  $\lambda=\otimes_p \lambda_p$. We often drop the index $p$ of $\lambda_p$ if there is no confusion.
Since there exists a section $\Phi \in I(s, \chi_{V})$ such that $\lambda(\varphi)=\Phi(g, s_0)$, hence one  write
 \begin{equation}
 E(g, s, \varphi)= E(g, s, \Phi).
 \end{equation}

The Siegel-Weil formula is extended by Kudla and Rallis (\cite{siegelweil} , \cite{KR2}), which asserts that two automorphic forms $I(g, \varphi)$ and $E(g, s_0, \varphi)$ are coincide:

\begin{theorem} (Siegel-Weil formula) \label{theo:Siegel-Weil}
Assume V is anisotropic or \\ $m -r>2$, where $r$ is the Witt index. Then for every $\varphi \in S(V(\A))$, the Eisenstein series  $E(g, s; \varphi)$ is holomorphic at $s_{0}$, and
\begin{center}
$E(g, s_{0}, \varphi)=\kappa I(g, \varphi)$,
\end{center}
 where $\kappa=2$ when $m\leq2$ and $\kappa=1$ otherwise.
\end{theorem}
 
 For the space which doesn't satisfy the above convergence condition, 
 one could study the regularized theta integral. Kudla and Rallis studied the regularized  Siegel-Weil formula \cite{regularized}, and a lot of cases have been proved by Gan, Qiu  and  Takeda in \cite{GQT}.

   In this paper, we study the quadratic space  $V=(M_{2}(\mathbb{Q}), Q)$ with quadratic form $Q=\det$. 
The main idea of proving the Siegel-Weil formula is to compare  the Fourier 
  coefficients of Eisenstein series and theta integral.  The next theorem shows that they are equal except the constant term. 
 \begin{theorem}{\bf Weak Siegel-Weil formula}\label{weakformula}
For any $\eta \in \mathbb{Q}^{*}$ and $\varphi \in S(V(\A))$,  $E_\eta(g, s, \varphi)$ is holomorphic at $s=s_0$ and $I_{\eta}(g, \varphi)$ is absolutely convergent. Moreover, one has
\begin{equation}
E_\eta(g, s_{0}, \varphi)=I_{\eta}(g, \varphi). 
\end{equation}
where $E_\eta(g, s_{0}, \varphi)$ is $\eta$-th Fourier coefficient of Eisenstein series, and $I_{\eta}(g, \varphi)$
 is $\eta$-th Fourier coefficient of $I(g, \varphi)$. 

\end{theorem}

\begin{remark}

 (1)\quad 
 The Siegel-Weil formula  holds if and only if
  $$E_{\eta}(g, s_{0}, \varphi)=I_{\eta}(g, \varphi)$$ for all $\eta \in \Q$. So  above theorem is almost the Siegel-Weil formula except the constant term, and we call it weak Siegel-Weil formula. 
 
 (2)\quad Kudla's work \cite{KuIntegral} shows that the Fourier coefficients of theta integral always have geometric explanations, i.e., degrees of cycles on Shimura varieties. Following the above theorem, we could compute these numbers via 
Fourier coefficients of Eisenstein series, which could be written as infinite products. It is also holds for any other indefinite  quaternions, since there exists the Siegel-Weil formula.
\end{remark}

 We drop rank one elements in $V(\mathbb{Q})$, and define
\begin{equation}
\widetilde{\theta}(g , h , \varphi)=\sum_{x\in V(\mathbb{Q}), \rank (x)\neq 1} \omega(g, h)\varphi(x).
\end{equation}
Then the integral is given by
\begin{equation}\widetilde{I}(g, \varphi)= \int_{[H]}\widetilde{\theta}(g, h, \varphi) dh.
\end{equation}

When $\eta \neq 0$, it is easy to see that  
  \begin{equation}
  \widetilde{I}_{\eta}(g, \varphi)=I_{\eta}(g, \varphi),
  \end{equation}
 where $\widetilde{I}_{\eta}(g, \varphi)$ is the Fourier coefficients of $\widetilde{I}(g, \varphi)$. 

\begin{theorem} \label{result}

When $\varphi_{\infty}=\varphi_{\infty}^{sp}$ as defined in equation (\ref{equsplit}), $\tilde{I}(g, \varphi)$ is absolutely convergent. Moreover, if $\Phi_{1}(g, s_0)=0$,  we have
$$\tilde{I}(g, \varphi)=E(g, s_{0}, \varphi).$$ 
Here \begin{equation}
\Phi_{1}(g, s)=\int_{\mathbb{A}}\Phi(wn(b)g, s)db,
\end{equation}
with $\Phi(g, s_0)=\lambda(\varphi)(g)$.

\end{theorem}

The above result could be extended to the case when $\varphi_\infty$ is a polynomial times a Gaussian. 

We assume that $D>0$ is a square free integer, and let $B=B(D)$ be the quaternion algebra which is ramified at a finite prime $p$ if and only if $p|D$.  The reduced
 norm, denoted by $\det$ in this paper,   gives a canonical quadratic form $Q$ on $B$ and makes it as a quadratic space. 
 When $D=1$, $B(D)=M_{2}(\Q)$.
In this paper, we denote any other quaternion algebra over $\mathbb{Q}$ by $V^{\prime}$, which is anisotropic.

  For a positive integer $N$ which is prime to
 $D$, let $\OO_D(N)$ be an Eichler order in $B$ of conductor $N$. We can view $L=(\OO_D(N), \det)$ as
 an even integral lattice in $V$. 

When $B$ is definite,
 there is a very interesting  question to compute the
 representation number (for a positive integer $m$)
 $$
r_L(m)=|\{ x \in \OO_D(N):\, \det x =
m\}|.
 $$
In general, it is very hard to compute, so we consider its average over the genus,
  which is denoted by
\begin{equation}
r_{D, N}(m) =r_{\gen(L)}(m) =\bigg(\sum_{L_1 \in \gen(L)}
\frac{1}{|\Aut(L_1)|}\bigg)^{-1} \sum_{L_1 \in \gen(L)}
\frac{r_{L_1}(m)}{|\Aut(L_1)|}.
\end{equation}
It depends only on $D$ and $N$,  and is independent of
the choice of Eichler order $\OO_D(N)$.
From Siegel's formula \cite{siegelformula}, it could be written as an infinite product. Now it is could be viewed as coefficients of Eisenstein series, which is the motivation of Siegel-Weil formula.

When $B$ is an indefinite quaternion,
let   $\Gamma_0^D(N) =\OO_D(N)^1$ be the  group of (reduced) norm $1$ elements in $\OO_D(N)$ and   let  $X_0^D(N) = \Gamma_0^D(N) \backslash  \mathfrak{H}$ be the associated Shimura curve.  For a positive integer $m$, let $T_{D, N}(m)$ be the Hecke correspondence on $X_0^D(N)$ which is defined in Section \ref{sect:Preli}.
Then  we define the normalized degree by
\begin{equation}\label{nordegree}
r_{D, N}(m)= -\frac{2}{ \vol(X_0^D(N), \Omega_0)}\deg T_{D,  N}(m),
\end{equation}
where
 $$
 \vol(X_0^D(N), \Omega_0) = \int_{X_0^D(N)} \Omega_0
 $$
is the  volume of $X_0^D(N)$ with respect to $\Omega_0 =\frac{1}{2\pi} y^{-2} dx \wedge dy$.
 
Two kinds of numbers $r_{D, N}(m)$ defined as above are Fourier coefficients of theta integral, see Section \ref{application} for details. 
 We could compute these numbers via Fourier coefficients of Eisenstein series in the following result.  
  
\begin{theorem}\label{rDN}
Let notations be as above and $k$ be the number of prime factors of $m$, then one has
\begin{eqnarray}
r_{D,N}(m)
&=&(-1)^{k+1}24m\prod_{p \nmid ND}\frac{p-p^{-\ord_pm}}{p-1}\nonumber\\
&&\times \prod_{p \mid N}\frac{2p-p^{-(\ord_pm-1)}-p^{-\ord_pm}}{p^{2}-1}\prod_{p \mid D}\frac{1}{(p-1)p^{\ord_pm}}.\nonumber
\end{eqnarray}
\end{theorem}
\begin{remark}
The case $D=1$ follows from the weak Siegel-Weil formula(Theorem \ref{weakformula}).
\end{remark}
As an application of  the Theorem \ref{rDN}, we obtain the following result.
\begin{corollary}
Assume $D$ is a square-free positive integer with even number of prime factors, one has
\begin{eqnarray}\label{degree}
&&\deg T_{D, N}(m)\nonumber\\&=&2mND\prod_{p \nmid ND}\frac{p-p^{-\ord_pm}}{p-1}\prod_{p \mid N}\frac{2-p^{-\ord_pm}-p^{-\ord_pm-1}}{p-1}\prod_{p \mid D}\frac{1}{p^{\ord_pm+1}}.\nonumber
\end{eqnarray}
\end{corollary}
When $D=1$, set $$M(N, m)=\{\kzxz{a}{b}{c}{d} \in M_{2}(\Z): ad-bc=m, c\equiv 0(\mod N)\}.$$  It could be written as a disjoint union $$ M(N, m)=\bigsqcup_{i=1}^K \Gamma_{0}(N)\alpha_i.$$ The Hecke operator $T(n)$ is a map of $Div(X_0(N))$ to itself given by :
 $$T(n)([\tau])=\sum_{i}^K [\alpha_i \tau],$$ where $\tau\in \H^{\ast}=\H \bigcup\{\Q\}$ and $[\tau]$ is the corresponding member of $X_0(N)$.
 Then one has
\begin{equation}
K=\frac{1}{2}\deg T_{D, N}(m)=mN\prod_{p \nmid N}\frac{p-p^{-\ord_pm}}{p-1}\prod_{p \mid N}\frac{2-p^{-\ord_pm}-p^{-\ord_pm-1}}{p-1}.
\end{equation}
Moreover, when $N=1$, we recover the well known result \begin{equation}K=m\prod_{p }\frac{p-p^{-\ord_pm}}{p-1}=\sum_{d \mid m}d.\end{equation}
  
  When $D\neq 1$, it also gives a similar explanation for Shimura curve $X_0^D(N)$. By Theorem \ref{rDN}, one reprove the main results in \cite{DuYang} as follows.
\begin{theorem}\cite[Theorem 1.1,1.2,1.3 and 1.4] {DuYang}For primes $p, q$, let $Dpq$ be a square-free positive
integer,  and let $N$ be a positive integer
prime to $Dpq$. For every positive integer $m$, we have
\begin{equation}
-\frac{2}{q-1} r_{Dp, N}(m) + \frac{q+1}{q-1} r_{Dp, Nq}(m)
 =-\frac{2}{p-1} r_{Dq, N}(m) + \frac{p+1}{p-1} r_{Dq, Np}(m)
\end{equation}
and 
\begin{equation}
r_{Dp, N}(m)= -\frac{2}{p-1} r_{D, N}(m) + \frac{p+1}{p-1} r_{D, Np}(m).
\end{equation}

\end{theorem}
\begin{remark}
 In \cite{DuYang}, the second equation is proved with the assumption $D>1$. We extend it to the case $D=1$ here by the weak Siegel-Weil formula(Theorem \ref{weakformula}).
\end{remark}
For   positive integer $m$, let $r_k(m)$ denotes the number of representations of an integer $m$ as a sum of $k$ squares. 
Recall that $\sum_{m\geq 0}r_{k}(m)q^m$ is the $k$-th power of theta series, i.e.,
$\sum_{m\geq 0}r_{k}(m)q^m=(\theta(\tau))^k$ with $\theta(\tau)=1+q+\cdots+q^{n^2}+\cdots$, where $q=e^{2\pi i\tau}$ and $\tau \in \H$. 

By the circle method, Hardy \cite{Ha1} and Ramanujan \cite{Ra1} proved that 
\begin{equation}
r_s(m)=\rho_{s}(m)+O(m^{\frac{s}{4}}), ~s\geq 5,
\end{equation}
where
\begin{equation}
\rho_{s}(m)=\frac{\pi^{\frac{s}{2}}}{\Gamma(\frac{s}{2})}m^{\frac{s}{2}-1} \mathfrak{G}_{s}(m),
\end{equation} 
 which is called singular series. Here
\begin{equation}
\mathfrak{G}_{s}(m)=\sum_{k=1}^{\infty} A_k(m),
\end{equation}
and
$$A_k(m)= \sum_{h=1, (h,k)=1}^{k}(\frac{1}{k}\sum_{j=1}^ke^{2\pi i hj^2/k})^{s}e^{-2\pi imh/k}.$$


When $s=5, 6, 7, 8$, Hardy \cite{Ha1}\cite{Ha2} gave the exactly formula by 
\begin{equation}
r_{s}(m)=\rho_{s}(m).
\end{equation}
He also claimed that it is false when $s=2$ and $s>8$. Bateman proved that this conclusion holds when $s= 3, 4$ in \cite{Ba}. Comparing factors of singular series with local Whittaker functions, we give another proof in the last section.

By the Siegel-Weil formula, we reprove the following result as follows. 
 
\begin{theorem}(Three and Four squares Theorem )\label{threefoursquareth}
  Let $m>0$, and assume $-4m=dc^2$, where  $d$ is the discriminant of $\Q(\sqrt{-m})$. One has
\begin{equation}
r_4(m)=8\sum_{d \mid m, 4 \nmid d} d,
\end{equation}
and 
\begin{equation}\label{threefinalintr}
r_3(m)=\frac{24 h(d) }{w}(1-\chi_d(2))\sum_{l \mid c, (l, 2)=1}l \prod_{p \mid l}(1-\chi_{d}(p)p^{-1}),
\end{equation}
where $p$ runs over prime factors of $l$.
Here $h(d)$ is the class number, $\chi_{d}$ is the character associated to quadratic field $\Q(\sqrt{-m})$ and $w$ is the number of roots of unity.
\end{theorem}

For a discriminant $m >0$, define Hurwitz class number
$H(m)$ be the number of classes of positive definite quadratic forms of discriminant $-m$, which is given by
\begin{equation}
H(m)=\frac{2h(d)}{w}\sum_{l \mid f}l \prod_{p \mid l}(1-\chi_{d}(p)p^{-1}),
\end{equation}
where $-m=df^2$. When $m$ is not a discriminant, $H(m)=0$. This formula is similar as equation (\ref{threefinalintr}), one could obtain Hirzebruch and  Zagier's result \cite{HZ}, see Corollary \ref{HZcor}.

Since $A_k(m)$ is multiplicative in $k$, one could write $$\mathfrak{G}_{s}(m)=\sum_{k=1}^{\infty} A_k(m)=\prod_{p}S_p(m),$$ 
where $S_p(m)=\sum_{r=0}^{\infty}A_{p^r}(m)$.

In order to prove Theorem \ref{threefoursquareth}, we compute the local Whittaker functions.
Comparing these functions with $S_p(m)$, one has the following result.
\begin{theorem}\label{localequality} 
Let $B(2)$ be the quaternion algebra over $\Q$ with discriminant $2$ and $B^0(2)$ be the
trace zero subspace of $B(2)$.\\ 
 1)
  Let $\mathcal{L}=\Z i+\Z j+\Z k$ be the lattice in quadratic space $V=(B^0(2), Q)$, one has 
\begin{equation}S_{p}(m)=W_{ p}(\frac{1}{2}, m).
\end{equation}
2) Let $L=\Z+\Z i+\Z j+\Z k$ be the lattice in $V=(B(2), Q)$, one has
\begin{equation}S_{p}(m)=W_{ p}(1, m).
\end{equation}
Here lattices $\mathcal{L}$ and $L$ are given in Section \ref{foursquare}, and the normalized local Whittaker function $W_{ p}(\frac{1}{2}, m)$ and $W_{ p}(1, m)$ are defined by equation (\ref{norwhitt}).
\end{theorem}
\begin{remark}
It implies that local factors of $\mathfrak{G}_{s}(m)$ are equal to local Whittaker functions. We expect that it could be extended to the case $s=5, 6, 7, 8$ at least.
\end{remark}
 Recall taht $r_s(m)$ could be written as product of local Whittaker functions.
 As an application of the above theorem, we reprove the following result in Section \ref{sechardy}.
\begin{theorem} \cite{Ba}\label{thba}
\begin{equation}
r_{s}(m)=\rho_{s}(m), ~s=3, 4.
\end{equation}
\end{theorem}
\begin{remark}
We find that the circle method and Siegel-Weil formula are the same in this question. From Hardy and Bateman's work, the above equality is true when $2<s <9$.  For $s=5, 6, 7, 8$, we leave it to readers to check by the Siegel-Weil method.
\end{remark}

This paper is organized as follows.

In Section \ref{sect:Preli}, we introduce the Weil representation and recall Kudla's matching pairs.  In Section \ref{splitcase}, we prove the convergence of theta integral $\tilde{I}(g, \varphi)$ ( $\tilde{I}_{\eta}(g, \varphi)$) in Theorem \ref{integral conv}  ( Proposition \ref{fourierintegral}). Combining Proposition \ref{fourierintegral} with Theorem \ref{mainresult}, we prove the weak Siegel-Weil formula \ref{weakformula} in Section \ref{sec:mainresult}. In Section \ref{application}, we identify numbers $r_{D, N}$ with Fourier coefficients of Eisenstein series and give the exactly formula in Theorem \ref{rDN}. Finally, we give the exactly formula for representation number of four and three squares sum in Theorem \ref{threefoursquareth} in Section \ref{foursquare}.  Comparing the local Whittaker functions with local factors of singular series, we prove Theorem \ref{localequality}
 and \ref{thba} in the last section.

\section{Preliminaries } \label{sect:Preli}

  Let
\begin{center}
$( ,  ) :V\times V \rightarrow  \Q$
\end{center}
be a nondegenerate symmetric bilinear form on $V$, then $(V, Q)$ is called a quadratic space, where $Q(x)=\frac{1}{2}(x, x)$. Set $G=\SL_2$, $H=O(V)$ and let $\psi: \mathbb{A}/\mathbb{Q} \rightarrow \mathbb{C}^{\times}$ be the canonical unramified additive character such that $\psi_{\infty}(x)=e^{2\pi ix}$.
The local component $\psi_{p}$ of $\psi$ at a nonarchimedean place $p$ is unramified if it is trivial on $\Z_{p}$ but nontrivial on $\frac{1}{p} \Z_{p}$.  Let
$$
\chi_V(x) = (x, (-1)^{\frac{m(m-1)}{2}} \det V)_\A
$$
be the associated quadratic character, where  $\A=\A_{\Q}$ is the adelic ring of $\Q$ and $ (, )_{\A}$ is the Hilbert symbol of $\Q$.

There is a Weil representation $\omega=\omega_{\psi, V}$   of $O(V)(\mathbb{A})\times \SL_{2}(\mathbb{A})$ acts on $S(V(\mathbb{A}))$.
 We could view it locally. For each prime p, denote the local representation  $\omega_p = \omega_{\psi, V_p}$ of $O(V)(\Q_p)\times \SL_{2}(\Q_p)$ acts on $S(V_p)$, where $V_{p}=V\bigotimes_{\Q}\Q_{p}$. Concretely,
 the orthogonal group $ O(V)(\mathbb{A})$ acts on $ S(V(\mathbb{A}))$  linearly,
\begin{center}
$\omega (h)\varphi (x)=\varphi(h^{-1}x)$.
\end{center}
The $\SL_2(\A)$-action is determined by (see for example \cite{KuSplit})
\begin{eqnarray}\label{weilrep}
&\omega(n(b))\varphi(x)=\psi(bQ(x))\varphi(x), \nonumber\\
&\omega(m(a))\varphi(x)=\chi_V(x) \mid a\mid^{\frac{m}2} \varphi(ax), \\
&\omega(w)\varphi= \gamma(V)  \widehat{\varphi} = \gamma(V)\int_{V(\A)}\varphi(y)\psi((x, y))dy ,\nonumber
\end{eqnarray}
where for $a \in \A^\times$, $b \in \A$
\begin{center}
$n(b)=
\left(
  \begin{array}{cc}
    1 & b \\
     & 1 \\
  \end{array}
\right),
m(a)=
\left(
  \begin{array}{cc}
    a &  \\
     & a^{-1} \\
  \end{array}
\right)
,
w=\left(
  \begin{array}{cc}
     & 1 \\
     -1&  \\
  \end{array}
\right),$
\end{center}
 $dy$ is the Haar measure on $V(\A)$ self-dual with respect to $\psi ((x, y))$, and $\gamma(V)=\prod_{p} \gamma(V_p)=1$. Here $\gamma(V_p)$ is a $8$-th root of unity  associated to  the local Weil representation at $p$ (local Weil index).  Let  $P=NM$  be the standard Borel subgroup of $\SL_2$, where $N$ and $M$ are subgroups of $n(b)$ and $m(a)$, respectively.

\subsection{Introduction to quaternions}\label{quaternion}
 In this paper, we only consider the quaternion algebras over $\Q$.
Let $B$ be a quaternion $\Q$-algebra, then
\begin{center}
$B=\Q+\Q i+\Q j+\Q ij$,
\end{center}
where $i^{2}=a, j^{2}=b, i j=-j i$ and $a, b \in \Q^{\times }.$
We write $B=\{\frac{a, b }{\Q}\}$.
The map
\begin{center}
$\iota: x=x_{1}+x_{2}i +x_{3}j +x_{4}k \rightarrow \bar{x}=x_{1}-x_{2}i -x_{3}j -x_{4}k$
\end{center}
is called main involution. Set the reduced trace $\tr(x)=x+\bar{x}$ and reduced norm $\det (x)=x\bar{x}$.

For example, we have $M_{2}(\Q)=\{\frac{1, 1 }{\Q}\}$ and division ring $\H=\{\frac{-1, -1}{\R}\}$ of Hamiltonian. 
For $M_{2}(\Q)$, we take $\iota$ to be the involution sending
\begin{center}
$\left(
  \begin{array}{cc}
    x_{1} & x_{2} \\
     x_{3}& x_{4} \\
  \end{array}
\right)$
to
$\left(
  \begin{array}{cc}
    x_{4} & -x_{2} \\
     -x_{3}& x_{1} \\
  \end{array}
\right)$.
\end{center}
The reduced norm and trace is the standard determinant and trace of the matrix. 

Let $D>0$ be a square  free integer, and let $B=B(D)$ be the unique
 quaternion algebra of discriminant $D$ over $\Q$, i.e., $B$ is
 ramified at a finite prime $p$ if and only if $p|D$.  We denote
 quadratic space $V^{\prime}=(B, \det)$.  For a positive integer $N$ prime to
 $D$, let $\OO_D(N)$ be an Eichler order in $B$ of conductor $N$ such that
\begin{enumerate}
\item  When $p\nmid N$, $\OO_D(N)_p:=\OO_D(N) \otimes_\Z \Z_p$ is the maximal order of $B_p=B\otimes_{\Q}\Q_p$.

\item When $p\mid N$,   there is an identification $B_p \cong M_{2}(\Q_p)$ under which
$$
\OO_D(N)_p=\left\{ \kzxz{a}{b}{c}{d} \in M_{2}(\Z_p):\,  c \equiv 0 \pmod  p \right\}.
$$
\end{enumerate}
 We can view $L=(\OO_D(N), \det)$ as
 an even integral lattice in $V^{\prime}$. 

 The quaternion $B$ is definite if and only if $D$ has odd
 number of prime factors.
In this case, we consider 
the average representation number over the $\gen(L)$, which is defined by
\begin{equation}
r_{D, N}(m)  =\bigg(\sum_{L_1 \in \gen(L)}
\frac{1}{|\Aut(L_1)|}\bigg)^{-1} \sum_{L_1 \in \gen(L)}
\frac{r_{L_1}(m)}{|\Aut(L_1)|}.
\end{equation}
 Here $\gen(L)$  is the set of equivalence classes of lattices in the same genus of $L$. For details see Section \ref{application}.

 When  $B(D)$ is indefinite, i.e., $D$ has even number of prime factors, then the representation number does not make sense.
In this case,  $V^\prime$ is of signature $(2, 2)$. We fix an embedding of \begin{center}
$i: B \hookrightarrow B \bigotimes \R \cong M_{2}(\R) $,
\end{center} such that $B^\times $ is invariant under the automorphism  $x \mapsto x^*= {}^tx^{-1}$ of $\GL_2(\R)$. Let   $\Gamma_0^D(N) =\OO_D(N)^1$ be the  group of (reduced) norm $1$ elements in $\OO_D(N)$ and identify it with $i(\Gamma_0^D(N))$.
Let $X_0^D(N) = \Gamma_0^D(N) \backslash  \mathfrak{H}$ be the associated Shimura curve.  

For a positive integer $m$, let $T_{D, N}(m)$ be the Hecke correspondence on $X_0^D(N)$ defined by
\begin{equation}\label{eq3}
\begin{split}
T_{D, N}(m)= &\{([z_1], [z_2]) \in X_0^D(N) \times  X_0^D(N):\\
& z_1 = i(x)z_2  \hbox{  for some } x \in \OO_D(N),   \,   \det x =m\}.
\end{split}
\end{equation}
Define\begin{equation}\deg T_{D, N}(m) =\deg (T_{D, N}(m) \rightarrow X_0^D(N))\end{equation} under the  projection $([z_1], [z_2]) \mapsto [z_1]$.
Let $\Omega_0 =\frac{1}{2\pi} y^{-2} dx \wedge dy$ be the normalized differential on $X_0^D(N)$, and let
 $$
 \vol(X_0^D(N), \Omega_0) = \int_{X_0^D(N)} \Omega_0
 $$
be the  volume of $X_0^D(N)$ with respect to $\Omega_0$.
The same as equation (\ref{nordegree}),  one define the normalized  degree by 
\begin{equation}
r_{D, N}(m)= -\frac{2}{ \vol(X_0^D(N), \Omega_0)}\deg T_{D,  N}(m).\nonumber
\end{equation}

\subsection{Kudla's matching}\label{kudlamatch}

Let $V^{(1)}, V^{(2)}$ be two quadratic spaces with the same dimension and the same quadratic character $\chi$. By the following diagram
\begin{equation}
\setlength{\unitlength}{1mm}
\begin{picture}(60, 20)
\linethickness{1pt}
\put(0,18){$S(V^{(1)}(\mathbb{A}))$}
\put(0,0){$S(V^{(2)}(\mathbb{A}))$}
\put(18,18){ \vector(3,-1){25}}
\put(18,0){ \vector(3,1){25}}
\thicklines
\put(45,8){$I(s_{0},\chi)$}
\put(25,16){$\lambda_{V^{(1)}}$}
\put(25,6){$\lambda_{V^{(2)}}$}
\end{picture}
\end{equation}
 one knows that the image of $\lambda_{V^{(1)}}, \lambda_{V^{(1)}}$ are in the same space $I(s_{0},\chi).$

Recall the definition in \cite{KuIntegral} as follows.

\begin{definition} For a prime $p \le \infty$, $\varphi_p^{(i)} \in S(V_p^{(i)})$, $i=1, 2$,  are said to be matching if
$$
\lambda_{V_p^{(1)}} (\varphi_p^{(1)}) = \lambda_{V_p^{(2)}}(\varphi_p^{(2)}).
$$
$\varphi^{(i)}=\otimes_{p\leq \infty} \varphi_p^{(i)} \in S(V^{(i)}(\A))$ are said to be matching if they match at each prime $p$, $i=1, 2$.
\end{definition}
For such a
  matching  pair $(\varphi^{(1)}, \varphi^{(2)})$, one has the following identity:
\begin{equation} \label{eq:matching}
I(g, \varphi^{(1)})=I(g, \varphi^{(2)}).
\end{equation}
It implies that their Fourier coefficients are equal.

 Recall that $V=(M_{2}(\mathbb{Q}), \det)$, and we denote $L_0^{sp}=M_{2}(\Z_p)$,
$$
L_1^{sp}= \left\{ \kzxz{a}{b}{c}{d} \in M_{2}(\Z_p):\,  c \equiv 0 \pmod  p \right\}
$$
 and 
$$
L_2^{sp}= \left\{ \kzxz{a}{b}{c}{d} \in M_{2}(\Z_p):\,  c \equiv 0 \pmod  {p^2} \right\}.
$$

For the quadratic space  $V^{\prime}=(B(D), \det)$,
set
$$\varphi^{\prime}=\begin{cases}
\cha(\widehat{\mathcal{O}_D(N)}) \otimes\varphi^{ra}_{\infty} \in S(V^{\prime}(\mathbb{A})) & \ff V^{\prime}~is~definite,\\
\cha(\widehat{\mathcal{O}_D(N)}) \otimes\varphi^{sp}_{\infty} \in S(V^{\prime}(\mathbb{A}))  & \ff V^{\prime}~is~indefinite.
\end{cases}$$
 
 One can prove $(\varphi^{ra}_{\infty}, \varphi^{sp}_{\infty})$ is a matching pair, which will be given in Section \ref{application}.
 By \cite[Proposition 3.1]{DuYang}, we have the following result:
\begin{proposition}  \label{globalmatching} Assume $\varphi_{D}^{N}=\otimes_{p\leq \infty}\varphi_{p}  \in S(V(\mathbb{A}))$ satisfies the following conditions:

(1) \quad When $p =\infty$,   $\varphi_\infty=\varphi_\infty^{sp}$,

(2) \quad When $p \nmid DN \infty$,  $\varphi_p= \varphi_0^{sp} $,

(3) \quad When $ p\mid N$, $\varphi_p=\varphi_1^{sp}$,

(4) \quad When $ p\mid D$, $\varphi_p=\frac{-2}{p-1}\varphi_0^{sp}+\frac{p+1}{p-1}\varphi_1^{sp} $ which is constructed in \cite{DuYang}.
Then $(\varphi^{\prime}, \varphi_{D}^{N})$ is a matching pair. Here $$
\varphi_i^{sp} =  \cha(L_i^{sp}),  \quad  i=0, 1, 2.
$$
\end{proposition}





\section{Theta integral for $M_{2}(\Q)$}\label{splitcase}

For the quadratic space $(V, Q)$ with $\dim V= m$, let $G=\SL_2, H=O(V)$. The integral is absolutely convergent precisely when  V is anisotropic or  $\dim(V)-r> 2$, where $ r$ is the Witt index of V, i.e., the dimension of a maximal isotropic $\Q$-subspace of $V$.  The Eisenstein series may be not holomorphic at $s_{0}$, so the Siegel-Weil formula is not always true.  Kudla and Rallis \cite{regularized} proved the regularized Siegel-Weil formula for some space.

\subsection{Theta integral }
From now on, we denote space $(M_{2}(\Q), \det)$ by $V$.\
There is no Siegel-Weil for this space since there are too many singular (rank one) elements. 
We define
\begin{equation}
\widetilde{\theta}(g , h , \varphi)=\sum_{x\in V(\mathbb{Q}), \rank (x)\neq 1} \omega(g) \varphi(h^{-1}x),
\end{equation}
and
\begin{equation}\widetilde{I}(g, \varphi)= \int_{[H]}\widetilde{\theta}(g, h, \varphi)dh.
\end{equation}
Here dh ( the half of the Tamagawa measure) is the invariant measure on $[H]$ with $\vol([H], dh)=1$. 
 Notice that $\tilde{\theta}(g , h , \varphi)$ is $P(\mathbb{Q})$-invariant with parabolic subgroup $P$.

In general, the theta  integral
\begin{equation}
I(g, \varphi)=\int_{[H]} \theta(g, h, \varphi)dh
\end{equation}
is not convergent, even for the following case.

\begin{example}
Let
$e=\kzxz{1}{}{}{1}$, $\varphi_{f}=\cha(\widehat{M_{2}(\Z)})$,
and  $\varphi_{\infty}=e^{-2\pi Q(x)}$, where $\widehat{M_{2}(\Z)}=M_{2}(\Z)\otimes_{\Z} \widehat{\Z}$. Then
\begin{equation}
I(e, \varphi)=\int_{[H]} \theta(e, h, \varphi)dh\nonumber
\end{equation}
is not convergent since there are infinitely many elements with zero determinant. 
\end{example}

\begin{theorem}\label{integral conv}

Assume that $\varphi_{\infty}= \varphi^{sp}_\infty$,
then the integral
$$\tilde{I}(g,\varphi)=\int_{[O(V)]}\tilde{\theta}(g,h,\varphi) dh$$
is absolutely convergent for each $g \in \SL_2(\A)$.
\end{theorem}

\begin{proof}
It suffices to prove that the integral $$\int_{[SO(V)]}\tilde{\theta}(g,h,\varphi) dh$$
is convergent.

It is easy to see
$$\mathrm{SO}(V)=\{(h_1,h_2)\in \GL_2\times \GL_2:  \det(h_1)=\det(h_2)\}/ \mathbb{Q}^\times,$$
which acts on $V$ via
$$
(h_1, h_2)x = h_1 x h_2^{-1}.
$$ 
View $\mathrm{SL}_2$ as a subgroup of $\mathrm{SO}(V)$ by $h\mapsto (h,1)$,
 we get an exact sequence
$$
1\longrightarrow \mathrm{SL}_2 \longrightarrow \mathrm{SO}(V) \longrightarrow \mathrm{PGL}_2 \longrightarrow 1.
$$
Write
$$J(g,\varphi)=\int_{[\SL_{2}]}\widetilde{\theta}(g,h,\varphi) dh.$$
We need to verify that $J(g,\omega(h')\varphi)=J(g,\varphi)$
for any $h'\in \mathrm{SO}(V)(\mathbb A)$.
The reason is as follows.

We write
$$
\tilde{\theta}(g,h,\varphi)= \theta_0(g,h,\varphi)+ \theta_2(g,h,\varphi),
$$
where
$$
\theta_i(g,h,\varphi)= \sum_{x\in V, \rank(x)=i}\omega(g, h)\varphi(x).
$$
It is easy to see that $\theta_i(g,h,\varphi)$ is invariant under $\mathrm{SL}_2$.
So it suffices to verify $J_i(g,\omega(h')\varphi)=J_i(g,\varphi)$ for
$$
J_i(g,\varphi)=\int_{[\SL_{2}]}\theta_i(g,h,\varphi) dh.
$$

The case $i=0$ is trivial since
$$\theta_0(g,h,\varphi)= \omega(g)\varphi(0). $$
Now we treat the case $i=2$. It is easy to have
$$
\theta_2(g,h,\varphi)= \sum_{\gamma\in \SL_2}\sum_{\eta\in \Q^\times}
\omega(g,\gamma h)\varphi(x_\eta).
$$
Here $x_\eta\in V$ is any element of norm $\eta$.
Thus
\begin{eqnarray*}
J_2(g,\varphi)&=&\sum_{\eta\in \mathbb{Q}^\times} \int_{\mathrm{SL}_2(\mathbb A)} \omega(g,h)\varphi(x_\eta) dh.
\end{eqnarray*}
It suffices to check
\begin{eqnarray*}
\int_{\mathrm{SL}_2(\mathbb A)} \omega(g)\varphi(h^{-1} h_1^{-1}x_\eta h_2) dh = \int_{\mathrm{SL}_2(\mathbb A)} \omega(g)\varphi(h^{-1}x_\eta) dh
\end{eqnarray*}
for any $h_1,h_2\in \GL_2(\mathbb A)$ with $\det(h_1)=\det(h_2)$.
Denote $y= h_1^{-1}x_\eta h_2 x_\eta^{-1}$, which lies in $\mathrm{SL}_2(\mathbb A)$.
Then the left-hand side is equal to
\begin{align}
\int_{\mathrm{SL}_2(\mathbb A)} \omega(g)\varphi(h^{-1} yx_\eta) dh
&= \int_{\mathrm{SL}_2(\mathbb A)} \omega(g)\varphi( (y^{-1}h)^{-1}x_\eta) dh\\ \nonumber
&=\int_{\mathrm{SL}_2(\mathbb A)} \omega(g)\varphi(h^{-1}x_\eta) dh.
\end{align}
Now we have
\begin{eqnarray*}
 \int_{[SO(V)]}\tilde{\theta}(g,h,\varphi) dh
&=& \int_{\mathrm{SL}_2(\mathbb A)\mathrm{SO}(V)(\mathbb{Q})\backslash \mathrm{SO}(V)(\mathbb A)} \int_{[\SL_{2}]}\tilde {\theta}(g,h_1h',\varphi) dh_1 dh'\\
&=& \int_{[PGL_{2}]} J(g,\omega(h')\varphi) dh'\\
&=& \int_{[PGL_{2}]} J(g,\varphi) dh'\\
&=& \mathrm{vol}([PGL_{2}])\ J(g,\varphi).
\end{eqnarray*}

If we use the Tamagawa measures on $\mathrm{SO}(V)(\mathbb A)$ and $\mathrm{\SL}_2(\mathbb A)$, then the quotient measure gives
$$\mathrm{vol}([PGL_{2}])=2.$$

It suffices to prove that $J(g,\varphi)$ is absolutely convergent.
There is an open compact subgroup $K$ of $\mathrm{SL}_2(\mathbb A_f)$ acting trivially on $\varphi$.
It follows that
$$\tilde{\theta}(g,h K,\varphi)=\tilde{\theta}(g,h,\varphi).$$
Let $K_\infty=\SO(2)(\R)$ and $\varphi'=\omega(g)\varphi$,
which is still a Schwartz function on $V(\mathbb A)$, and the infinite part $|\varphi_\infty'|$ is bounded by a polynomial times a Gaussian.

 It is known that $\varphi_{\infty}(x k_\theta)=\varphi_{\infty}(x)$ for $k_\theta \in K_{\infty}$. The proof is reduced to show the absolute convergence of
$$J'(\varphi)=\int_{\mathrm{SL}_2(\mathbb{Q})\backslash \mathrm{SL}_2(\mathbb A)/KK_\infty}\tilde{\theta}(h,\varphi') dh.$$ 

Then
$$
J'(\varphi)=\int_{\Gamma\backslash \mathbb H}\tilde{ \theta}(h,\varphi') dh.
$$
Here $\mathbb H$ is the upper half plane, and $\Gamma= \mathrm{SL}_2(\mathbb{Q})\cap K$ is a subgroup of $\mathrm{SL}_2(\mathbb Z)$ with finite index.

Let
$$\Omega=\{x+yi\in \mathbb H: -1/2<x\leq 1/2, \ |x+yi|>1\}$$
be the standard fundamental domain of $\mathrm{SL}_2(\mathbb Z)\backslash \mathbb H$.
It suffices to prove that
$$
J''(\varphi)=\int_{\Omega} \tilde{\theta}(h,\varphi') dh
=\int_{\Omega}\tilde{ \theta}\bigg(
\kzxz{\sqrt{y}}{0}{0}{\frac{1}{\sqrt{y}}}
\kzxz{1}{x}{}{1}
,\varphi' \bigg) \frac{dxdy}{y^2}
$$
is absolutely convergent.

The integrand grows slowly as $y\to \infty$ using the decay of $\varphi'_\infty$.
Note that $h$ has only infinite part, and thus, essentially (up to a finite linear combination)
$$
\tilde{\theta}(h,\varphi') =\sum_{l\in L, \rank(l)\neq 1} \omega(h) \varphi'_\infty(l)
$$
where $L$ is some lattice in $V$.
Then it is easy to check that
\begin{equation}\label{equtheta}
 \bigg|\tilde{\theta}\bigg(
\kzxz{\sqrt{y}}{0}{0}{\frac{1}{\sqrt{y}}}
\kzxz{1}{x}{}{1}
,\varphi'\bigg) \bigg|
\leq C
\end{equation}
for some constant $C$.

For convenience, we just assume that $x=0$ and $L=M_{2 }(\Z)$.

Since $\varphi'_\infty$ is a Schwartz function, there exists
$M>0$ for $c>0$, when $\mid l_1\mid +\mid l_2\mid +\mid l_3\mid +\mid l_4\mid >M$,
$$l_1^2l_2^2l_3^4l_4^4 \mid \varphi'_\infty(l) \mid <c,$$
for any $l=\kzxz{l_1}{l_2}{l_3}{l_4} \in L$.
Thus if $$\mid \frac{l_1}{\sqrt{y}}\mid +\mid \frac{l_2}{\sqrt{y}}\mid +\mid \sqrt{y} l_3\mid +\mid \sqrt{y} l_4\mid >M,$$ one has 
$$y^2l_1^2l_2^2l_3^4l_4^4\mid \varphi'_\infty\bigg(\kzxz{\sqrt{y}}{0}{0}{\frac{1}{\sqrt{y}}}^{-1}l \bigg)\mid<c.$$
 Then is is easy to see
\begin{eqnarray}
&&\bigg|\sum_{l\in L, \rank(l)\neq 1, l_1l_2l_3l_4\neq 0} \varphi'_\infty\bigg(\kzxz{\sqrt{y}}{0}{0}{\frac{1}{\sqrt{y}}}^{-1}l \bigg)\bigg|\nonumber\\
&<& \frac{16c}{y^2}\sum_{l_1 \in \N} \frac{1}{l_1^2}\sum_{l_2 \in \N} 
\frac{1}{l_2^2}\sum_{l_3 \in \N} \frac{1}{l_3^4}\sum_{l_4 \in \N} \frac{1}{l_4^4} + \phi= \frac{16c \zeta(2)^2 \zeta(4)^2}{y^2} +\phi,\nonumber 
\end{eqnarray}
where $\phi$ is the sum over the subset $$B=\{l=\kzxz{l_1}{l_2}{l_3}{l_4}\in L, \rank(l)\neq 1 :\mid \frac{l_1}{\sqrt{y}}\mid +\mid \frac{l_2}{\sqrt{y}}\mid +\mid \sqrt{y} l_3\mid +\mid \sqrt{y} l_4\mid  \leq M\},$$
which is a finite subset. When $y >M^2$, $B$ is empty. So
$$
 \bigg|\sum_{l\in L, \rank(l)\neq 1, l_1l_2l_3l_4\neq 0} \varphi'_\infty\bigg(\kzxz{\sqrt{y}}{0}{0}{\frac{1}{\sqrt{y}}}^{-1}l \bigg)\bigg|
$$
is bounded by some constant. 

When $ l_1l_2l_3l_4=0$, we could follow the same method.
Then one obtains the equation (\ref{equtheta}), the result follows.
\end{proof}
\begin{remark}
When $\varphi_{\infty}$ is a polynomial times a Gaussian, the above theorem is true. We leave details to the readers. 
\end{remark}

\subsection{ Fourier coefficient of theta integral}

The $\eta$-th Fourier coefficient of $I(g, \varphi)$ is given by
\begin{eqnarray}I_{\eta}(g,  \varphi)&=&\int_{\mathbb{Q}\setminus \mathbb{A}}I(n(b)g, \varphi)\psi(-b\eta)db\nonumber  \\
&=&\int_{[H]}\theta_{\eta}(g, h, \varphi)dh
\end{eqnarray}
with $$\theta_{\eta}(g, h, \varphi)=\sum_{x\in V(\mathbb{Q})[\eta]}\omega(g)\varphi(h^{-1}x),$$
where $$V(\mathbb{Q})[\eta] =\{ x\in V(\mathbb{Q}):\,   Q(x)=\eta\}.$$

For $\eta\neq 0$ and for any choice of $\varphi$,
\begin{eqnarray}
\tilde{I}_{\eta}(g, \varphi)=I_{\eta}(g, \varphi),\nonumber
\end{eqnarray}
where $\tilde{I}_{\eta}(g, \varphi)$ is $\eta$-th Fourier coefficients  of $\tilde{I}_{\eta}(g, \varphi)$.
By Theorem \ref{integral conv}, one knows that $\tilde{I}_{\eta}(g, \varphi)$ 
 is absolutely convergent when $ \varphi_{\infty}= \varphi^{sp}_\infty$. We extend it to all $\varphi \in S(V(\mathbb{A}))$ in the following result.

\begin{proposition}\label{fourierintegral}
For any $\eta \in \mathbb{Q}^{\ast}$ and $\varphi=\otimes_{p} \varphi_{p} \in S(V(\mathbb{A}))$,  the integral
$I_{\eta}(g, \varphi)$
is absolutely convergent.
\end{proposition}
\begin{proof}

For convenience, we prove that $I_{\eta}(e, \varphi)$
is absolutely convergent. It is not hard to prove all cases if one replace $\varphi$ by $\omega(g) \varphi$. 

Up to a finite linear combination, we suppose that support$(\varphi_{f})$ $\subseteq \widehat{L}$, where L is a lattice in the space $V(\mathbb{Q})$. We could set $\varphi_{f}= \cha(\widehat{M_{2}(\mathbb{Z})})$.

Notice that the set $V(\mathbb{Q})[\eta] $  for $\eta\neq 0$ is in one orbit, so there is an bijective map
\begin{center}
$SO(V)(\mathbb{Q})_{x_{\eta}}\setminus SO(V)(\mathbb{Q}) \longleftrightarrow V(\mathbb{Q})[\eta],$
\end{center}
\begin{center}
$h \rightarrow hx_{\eta}$,
\end{center}
where $x_{\eta}$ is any element in $V(\Q)$ with $\det x_{\eta}=\eta $.
\begin{eqnarray}
I_{\eta}(e, \varphi)&=&\frac{1}{2}\int_{[SO(V)]} \sum_{x\in V(\mathbb{Q})[\eta]} \varphi(h^{-1}x)dh \nonumber \\ &=&\frac{1}{2}\int_{SO(V)(\mathbb{Q})_{x_{\eta}}\setminus SO(V)(\mathbb{A})} \varphi(h^{-1}x_{\eta})dh \nonumber\\
 &=& \frac{1}{2}\vol(SO(V)_{x_{\eta}})\int_{SO(V)(\mathbb{A})_{x_{\eta}}\setminus SO(V)(\mathbb{A})} \varphi(h^{-1}x_{\eta})dh \nonumber\\
&=& \prod_{p\leq \infty}\int_{SO(V)(\mathbb{Q}_{p})_{x_{\eta}}\setminus SO(V)(\mathbb{Q}_{p})} \varphi_{p}(h^{-1}x_{\eta})dh_{p}.
\end{eqnarray}
It is easy to obtain that
\begin{equation}
SO(V)=\{(h_{1}, h_{2})\mid  h_{1}, h_{2} \in \GL_{2}(\mathbb{Q}), \det(h_{1})=\det(h_{2})\}/\mathbb{Q}^{\times}.
\end{equation}
Let $h=(h_{1}, h_{2})\in SO(V)$,  the action is  given by $h\cdot x=h_{1}xh_{2}^{-1}$.
For any $\eta \in \mathbb{Q }\setminus \mathbb{Z}$, it is easy to get that
$$\prod_{p\leq \infty}\int_{SO(V)(\mathbb{Q}_{p})_{x_{\eta}}\setminus SO(V)(\mathbb{Q}_{p})} \varphi_{p}(h_{p}^{-1}x_{\eta})dh_{p}= 0.$$

Let $\eta \in \mathbb{Z}, \eta\neq 0$, assume that 
$x_{\eta}=\kzxz{\eta}{}{}{1}$, then
\begin{center}
$SO(V)_{x_\eta}=\Big \{(h_{1}, h_{2}) |h_{1}=\kzxz{a}{b}{c}{d}, 
h_{2}=\kzxz
  {a}{ \frac{1}{\eta}b} {\eta c}  {d}, h_{1}, h_{2} \in \GL_{2}\Big \}/\mathbb{Q}^{\times}.$
\end{center}
Hence $SO(V)_{x_\eta}\setminus SO(V) \cong 1\times \SL_{2}$, and the integral
\begin{eqnarray}\label{coeff comp}
I_{\eta, p}(e, \varphi_p)&=&\int_{SO(V)(\mathbb{Q}_{p})_{x_{\eta}}\setminus SO(V)(\mathbb{Q}_{p})} \varphi_{p}(h^{-1}x_{\eta})dh_{p}\nonumber\\
&= &\int_{\SL_{2}(\mathbb{Q}_{p})} \varphi_{p}(x_{\eta}h_{p})dh_{p}.
\end{eqnarray}

When $p<\infty$, $I_{\eta, p}(e, \varphi_p)=\vol(A_{\eta}),$
where
\begin{center}
$A_{\eta}=\bigg\{\left(
  \begin{array}{cc}
  a & b \\
  c & d \\
   \end{array}
   \right) | a, b \in \frac{1}{\eta}\Z_{p}, c, d \in \Z_{p}\bigg\}$,
\end{center}
which is a compact subset of $\SL_{2}(\mathbb{Q}_{p})$, and $\vol(A_{\eta})\leq |\frac{1}{\eta}|_{p}^{2}$ .
Now we assume $\eta >0$, and we have the following estimation
\begin{equation}\label{finite part}
\prod_{p<\infty}I_{\eta, p}(e, \varphi_p) \leq\eta^{2}.
\end{equation}

When $p=\infty$,

\begin{eqnarray}
I_{\eta, \infty}(e, \varphi_{\infty})&=&\int_{SO(V)(\mathbb{R})_{x_{\eta}}\setminus SO(V)(\mathbb{R})} \varphi_{\infty}(h^{-1}x_{\eta})dh_{\infty}\nonumber\\
&= &\int_{\SL_{2}(\mathbb{R})} \varphi_{\infty}(x_{\eta}h_{\infty})dh_{\infty}.\nonumber
\end{eqnarray}
One has the Iwasawa decomposition of group $\SL_{2}(\mathbb{R})$
\begin{equation}
\SL_{2}(\mathbb{R})=N(\mathbb{R})M(\mathbb{R})\SO(2)(\mathbb{R}).\nonumber
\end{equation}
By the \cite[pp194, Lemma 4]{Weil}, there exists
$\phi_{\infty} \in S(V(\R))$ such that $\mid\varphi_{\infty}(x k_\theta)\mid \leq \phi_{\infty}$ for $k_\theta \in \SO(2)(\mathbb{R})$.
 So the integral 
\begin{eqnarray}
\mid I_{\eta, \infty}(e, \varphi_{\infty})\mid &\leq&\int_{0}^{2\pi}\int_{\mathbb{R}\times \mathbb{R}_{+}^{*}} \mid\varphi_{\infty}(x_{\eta}n(x)m(y^{\frac{1}{2}})k_{\theta})\mid \frac{dxdy}{y^{2}}d\theta \nonumber\\
&\leq & 2\pi \int_{\mathbb{R}\times \mathbb{R}_{+}^{*}}\phi_{\infty}\Bigg(\left(
  \begin{array}{cc}
  \eta y^{\frac{1}{2}} & \eta xy^{\frac{-1}{2}}\\
  0 & y^{\frac{-1}{2}} \\
   \end{array}
   \right)\Bigg)\frac{dxdy}{y^{2}}\nonumber \\
   &=& 2\pi \int_{\mathbb{R}\times \mathbb{R}_{+}^{*}}\phi_{\infty}\Bigg(\left(
  \begin{array}{cc}
  \eta y^{\frac{-1}{2}} & \eta xy^{\frac{1}{2}}\\
  0 & y^{\frac{1}{2}} \\
   \end{array}
   \right)\Bigg)dxdy.
\end{eqnarray}

Since $\phi_{\infty}$ is a Schwartz function, there exists a constant $M>0$ for $c>0$, such that
  \begin{center}
  $ \phi_{\infty}\bigg(\left(
  \begin{array}{cc}
  \eta y^{\frac{-1}{2}} & \eta xy^{\frac{1}{2}}\\
  0 & y^{\frac{1}{2}} \\
   \end{array}
   \right)\bigg) < \frac{c}{(1+(\eta y^{\frac{-1}{2}})^{2}(\eta xy^{\frac{1}{2}})^{2})(1+(y^{\frac{1}{2}})^{4})}=\frac{c}{(1+y^{2})(1+\eta ^{4}x^{2})}$
  \end{center}
when $\mid  \eta y^{\frac{-1}{2}}\mid + \mid \eta xy^{\frac{1}{2}}\mid +\mid y^{\frac{1}{2}}\mid > M.$
Then
\begin{align}
\begin{split}
\mid I_{\eta, \infty}(e, \varphi_{\infty})\mid &\leq 2\pi \int_{\mathbb{R}\times \mathbb{R}_{+}^{*}} \frac{c}{(1+y^{2})(1+\eta^{4}x^{2})} dxdy \\
 &+ 2\pi \int_{B}\phi_{\infty}\Bigg(\left(
  \begin{array}{cc}
  \eta y^{\frac{-1}{2}} & \eta xy^{\frac{1}{2}}\\
  0 & y^{\frac{1}{2}} \\
   \end{array}
   \right)\Bigg)dxdy.\nonumber
\end{split}
\end{align}
Here $$B=\{ (x, y) \in\mathbb{R}\times \mathbb{R}_{+}^{*} \mid \mid  \eta y^{\frac{-1}{2}}\mid + \mid \eta xy^{\frac{1}{2}}\mid +\mid y^{\frac{1}{2}}\mid \leq M \} ,$$ which is contained in a compact set.
Thus the integral $I_{\eta, \infty}$ is convergent, so the integral $I_{\eta}(e, \varphi)$ is absolutely convergent from equation (\ref{finite part}).
\end{proof}

\section{Weak Siegel-Weil formula}\label{sec:mainresult}

In this section, we prove Theorems \ref{weakformula} and \ref{result}.
\subsection{Eisenstein series}

For the finite place $p$, we let $K_{p}=\SL_2(\Z_{p})$, a maximal compact subgroup of $G_{p}=G(\Q_{p})$. If $p$ is infinite place, we set $K_{\infty}= \SO(2)(\R)$. We let $K=\prod_{p}K_{p}$ be the maximal compact subgroup of $G(\A)$, and we have $G(\A)=P(\A)K$, 
where $P=N M$ is the maximal parabolic subgroup of $G$ (called Siegel parabolic).

The induced representation
 $I(s, \chi_V)= \Ind_{P}^{G}( | | ^{s}\chi_V)$
of $G(\A)$ consists of smooth functions $\Phi(g, s)$ on $G(\A)$ such that
\begin{equation} \label{eq:1.1}
\Phi(nm(a)g, s) =
\chi_{V}(a)|a|^{s+1}\Phi(g, s).
\end{equation}
The Eisenstein series is defined by
\begin{equation}
E(g, s, \Phi)= \sum_{\gamma\in P \setminus G}
\Phi(\gamma g, s).
\end{equation} 
It is absolutely convergent for $\Re(s)> 1$.

There is a $G(\A)$-intertwining map
\begin{equation}
\lambda=\lambda_{V} : S(V(\mathbb{A})) \rightarrow I(s_{0}, \chi_{V}), \quad \lambda(\varphi)(g)= \omega(g)\varphi(0).
\end{equation}
 There exists a section $\Phi \in I(s, \chi_{V})$ such that $\lambda(\varphi)=\Phi(g, s_0)$, and one could write
 \begin{equation}
 E(g, s, \varphi)= E(g, s, \Phi).
 \end{equation}

Now we consider the case for $V=(M_{2}(\Q), \det)$. By \cite[Lemma 1.3]{regularized} and \cite[Theorem 4.12]{regularized}, we  know that $E(g, s, \Phi)$ at most has one simple pole at $s_{0}$. So one has the Laurent expansion

 \begin{equation}\label{Laurent}
 E(g, s, \Phi)=\frac{A_{-1}(g, \Phi)}{s-s_{0}}+A_{0}(g, \Phi)+O(s-s_{0}).
\end{equation}

Assume that $\Phi(s)=\bigotimes_{p}\Phi_{p}(s)$ is a factorizable standard section of $I(s, \chi)$. For $\eta \neq 0$, $\Re(s)>1$, the $\eta$-th Fourier coefficient of $E(g, s, \Phi)$ is
\begin{eqnarray}
E_{\eta}(g, s, \Phi)&=&\int_{\mathbb{A}/\mathbb{Q}}E(n(b)g, s, \Phi)\psi(-b\eta)db\nonumber\\
&=&\prod_{p}W_{\eta, p}(g, s, \Phi_{p}),
\end{eqnarray}
where
\begin{eqnarray}\label{whittakerfun}
W_{\eta, p}(g, s, \Phi_{p})=\int_{\mathbb{Q}_{p}}\Phi_{p}(w n(b)g, s)\psi_{p}(-b\eta)db.
\end{eqnarray}
The integral $W_{\eta, p}(g, s, \Phi_{p})$ extends to an entire function of $s$ \cite{Karel}, \cite{Wallach}. For any fixed $s$, it defines an element of the one dimensional space
$$\Hom_{G_{p}}(I_{p}(\chi_V,s), \Ind_{N_{p}}^{G_{p}}(\psi_{\eta})).$$

From \cite[Section 2]{regularized}, one has:
\begin{lemma}\label{eiscoeff}
For every $\eta \in \mathbb{Q}^{\ast}$, $E_{\eta}(g, s, \Phi)$ is holomorphic at $s_{0}$.
\end{lemma}


\subsection{Weak Siegel-Weil formula}

If $G$ is a algebraic group over number field $F$, $G_{A_{F}}$ is a topological locally compact group. If $\omega$ is a gauge form on $G$, and $(\lambda_{\calP})$ \cite[Chapter 2]{Weil1} is a set of convergence factor of $G$, then the Tamagawa measure $\Omega=(\omega, (\lambda_{\calP}))$ is a left invariant measure on $G_{A_{F}}$, where $\calP$ is any place of $F$. This measure is independent of the choice of $\omega$, and is called the Tamagawa measure derived from the convergence factors $(\lambda_{\calP})$. If $(1)$ is the convergence factors, the measure $\Omega=(\omega, (1))$ is called Tamagawa measure for $G$, and the number \begin{equation}\tau(G)=\int_{G_{A_{F}}/ G_{F}}(\omega, (1))
\end{equation}
 is the Tamagawa number of $G$.

Now we consider the fixed space $V=M_{2}(\Q)$. Choosing the Tamagawa measure $dh^{\prime}$ on $O(V)(\mathbb{A})$ which is 2 times $dh$. The gauge form $\omega= dx_{1}\wedge dx_{2}\wedge dx_{3}\wedge dx_{4}$ on $V$ determines a measure
$\omega _{p}$ for $V_{p}$ , which is the self dual for the pairing $[x, y]= \psi_{p}((x, y)),$ where $( , )$ is the bilinear form associated to $Q$.
On the other hand, the gauge form $\alpha= d\eta$ on $\mathbb{Q}$ determines a measure $ \alpha_{p}= d_{p}\eta$, which is self -dual with respect to the pairing $[b, \eta]= \psi_{p}(b\eta)$.

We can split $\omega = \omega_{\eta} \wedge \alpha$  \cite[section 2.5]{Hida}, where $\omega_{\eta}$ is the gauge form on $V[\eta]$. Since
\begin{center}
 $O(V)_{x_{\eta}}\backslash O(V) \cong V[\eta]$,
\end{center}
$\omega_{\eta}$ is also the gauge form on $O(V)_{x_{\eta}}\backslash O(V)$.
Then $dh^{\prime}$ induces the Tamagawa measure $\omega_{\eta}$ on $ O(V)_{x_{\eta}}(\mathbb{A})\backslash O(V)(\mathbb{A})$ for $x_{\eta}\in V$ with $Q(x_{\eta})= \eta$ ( by the uniqueness of the Tamagawa measure).
We now compute the Fourier coefficient of the theta integral
\begin{eqnarray}\label{thetacoeff}
I_{\eta}(g, \varphi)&=& \int_{[H]}\theta_{\eta}(g, h, \varphi)dh\nonumber\\ &= &\int_{O(V)(\mathbb{Q})_{x_{\eta}}\setminus O(V)(\mathbb{A})} \omega(g)\varphi(h^{-1}x_{\eta})dh \nonumber\\
& =& \frac{1}{2} \tau(O(V)_{x_{\eta}}) \int_{ O(V)_{x_{\eta}}(\mathbb{A})\backslash O(V)(\mathbb{A}) }\omega(g) \varphi (h^{-1}x_{\eta}) \omega_{\eta} \nonumber\\
& =&\prod_{p} O_{\eta, p}(\omega(g_{p}) \varphi_{p}). 
\end {eqnarray}
Where $ \tau(O(V)_{x_{\eta}}) $ is the Tamagawa number of $O(V)_{x_{\eta}}$, Weil have proved that $ \tau(O(V)_{x_{\eta}}) = 2$.
In the last step, we have assumed that $ \varphi$ is factorizable and written
\begin{equation}
O_{\eta, p}(\varphi_{p})= \int_{O(V)_{x_{\eta}}(Q_{p})\setminus O(V)(Q_{p})} \varphi_{p}(h^{-1} x_{\eta}) \omega_{\eta, p},\nonumber
\end{equation}
$p \leq \infty,$ for the local orbital integral.

When $ \varphi$ is factorizable, then  the associated $\Phi(s)=\bigotimes_{p} \Phi_{p}(s)$ is also factorizable, where $\Phi_{p}(s) \in I_{p}(s, \chi)$.
For $\eta \in \mathbb{Q}^{\times}$ and $\Re(s)> 1$, the $\eta$-th Fourier coefficient of $E(g, s, \varphi)$ is
\begin{equation}
E_{\eta}(g, s, \varphi)= \int_{\mathbb{Q}\setminus \mathbb{A}} E(n(b)g, s, \varphi) \psi_{-\eta}(b)db = \Pi_{p} W_{\eta, p}(g_p, s, \varphi_{p})\nonumber
\end{equation}
with
\begin{equation}
W_{\eta, p}(g_p, s, \varphi_{p}) = \int_{\mathbb{Q}_{p}} \Phi_{p}( w n(b)g_p, s) \psi_{-\eta}(b)db.\nonumber
\end{equation}
$O_{\eta, p}(\omega_p(g_p)\varphi_{p})$ and $ W_{\eta, p}(g_p, s_0, \varphi_{p})$ are two distributions. By \cite[Proposition 4.2]{Rallis},  
\begin{equation}
E_{\eta}(g, s_{0}, \varphi)=cI_{\eta}(g, \varphi),
\end{equation}
where  $c$ is a constant.

Combining the Proposition \ref{fourierintegral} with the following result, we could prove Theorem \ref{weakformula}.
\begin{theorem}\label{mainresult}
For any $\eta \in \mathbb{Q}^{*}$, we have
\begin{equation}E_\eta(g, s_{0}, \varphi)=I_{\eta}(g, \varphi)
\end{equation}
for all $\varphi \in S(V(\mathbb{A}))$ and $g \in \SL_{2}(\mathbb{A}).$
\end{theorem}
\begin{proof}
It suffices to prove that the constant $c=1$.
 We could choose functions $\varphi$ such that $\varphi_{\infty}$ has compact support. Let $\lambda(\varphi_{p})=\Phi_{p}(s_0)$, then
\begin{eqnarray}\label{eisenstein coeff}
&&E_\eta(e, s_{0}, \varphi)=\prod_{p \leq \infty}W_{\eta, p}(e, s_0, \Phi_{p})\nonumber\\ 
&&=  \prod_{p \leq\infty}\int _{\mathbb{Q}_{p}} \int_{V(\mathbb{Q}_{p})} \psi_{p}(b Q(y)) \varphi_{p}(y) \omega_{ p}  \cdot  \psi_{p}(-\eta b) d_{p}b \nonumber\\
&&=  \prod_{p \leq\infty} \int _{\mathbb{Q}_{p}} \int_{\mathbb{Q}_{p}} \psi_{p}(b u) M_{\varphi_{p}}(u) d_{p}u \cdot  \psi_{p}(-\eta b) d_{p}b \nonumber\\
&&=\prod_{p \leq\infty} \int_{\mathbb{Q}_{p}}  \widehat{M_{\varphi_{p}}}(b)  \psi_{p}(- \eta b) d_{p}b.
\end{eqnarray}
Here
\begin{center}
$ M: S(V(\mathbb{Q}_{p}))\rightarrow S(\mathbb{Q}_{p}), \varphi_{p} \mapsto M_{\varphi_{p}}$
\end{center}
is the map defined by integration over the fibers with respect to the measure determined by the restriction of the gauge form $\omega_{\eta}$. 

For archimedean case, the function $\widehat{M_{\varphi_{\infty}}}$ has a compact support.
For the finite place, the function $\widehat{M_{\varphi_{p}}}$ lies in the Schwartz space
$S(\mathbb{Q}_{p})$. So one has
\begin{eqnarray}
 \int_{\mathbb{Q}_{p}}  \widehat{M_{\varphi_{p}}}(b) \psi_{p}(- \eta b) d_{p}b 
=   M_{\varphi_{p}}(\eta)= O_{\eta, p}(\varphi_{p}).\nonumber
\end{eqnarray}
In the above equation, we used the same measure given by $ \omega_{\eta}$ on $$O(V)_{x_{\eta}}(\mathbb{Q}_{p}) \setminus O(V)(\mathbb{Q}_{p}) \cong V_{p}[\eta].$$

Combining it with equations (\ref{thetacoeff}) and (\ref{eisenstein coeff}), one obtains
\begin{eqnarray}
E_{\eta}( e, s_{0}, \varphi) =   \prod_{p < \infty} O_{\eta, p}(\varphi_{p})  O_{\eta, \infty}(\varphi_{\infty}) 
=  I_{\eta}(e, \varphi).\nonumber
\end{eqnarray}
So the constant $c=1$, thus we get the result.
\end{proof}

\subsection{Relations between theta integral and Eisenstein series}
Recall that \begin{equation}
 E(g, s, \Phi)=\frac{A_{-1}(g, \Phi)}{s-s_{0}}+A_{0}(g, \Phi)+O(s-s_{0}).\nonumber
\end{equation}

Define \begin{equation}
\Phi_{1}(g, s)=\int_{\mathbb{A}}\Phi(wn(b)g, s)db,
\end{equation}
and the Laurent series is given by
\begin{equation} \label{constant residue}
\Phi_{1}(g, s)=\frac{A_{-1}(g, \Phi)}{s-s_{0}}+B_{0}(g, \Phi)+O(s-s_{0}).
\end{equation}

We identify  notation $\varphi$ with $\Phi$ if $\Phi(g, s_0)=\lambda(\varphi)$.

Now we consider the constant term, i.e., the coefficient with $\eta=0$.
The constant term of Eisenstein series is 
\begin{eqnarray}\label{constant term}
E_{P}(g, s, \Phi)&=&\int_{\mathbb{A}/\mathbb{Q}}E(n(b)g, s, \Phi) db\\
&=&\Phi(g, s)+\Phi_{1}(g, s).\nonumber
\end{eqnarray}

\begin{proof}[\bf{Proof of Theorem \ref{result}}]
Combining it with the Theorem \ref{integral conv}, it suffices to prove $\tilde{I}(g, \varphi)=E(g, s_{0}, \varphi).$

Firstly, we consider the constant term of both sides
\begin{equation}
\tilde{I}_{P}(g, \varphi)= \int_{[H]}\theta_{0}(g, h, \varphi) dh=\omega(g)\varphi(0).
\end{equation}
Since $\Phi_{1}(g, s_0)=0$, by equation (\ref{constant term}),
one obtains $$E_{P}(g, s_0, \varphi)=\Phi(g, s_0)=\omega(g)\varphi(0)=\tilde{I}_{P}(g, \varphi).$$
Combining it with the Theorem \ref{mainresult}, we know all the Fourier coefficients are equal. By the next lemma, one obtains the result.
\end{proof}

\begin{lemma}
Let $f(g)=\tilde{I}(g, \varphi)-E(g, s_{0}, \varphi)$, then $f=0$.
\end{lemma}
\begin{proof}
Notice that $f$ is a function with all Fourier coefficients being $0$.
Set
\begin{center}
$F(x)=f\bigg(\left(
                              \begin{array}{cc}
                                1 & x \\
                                & 1 \\
                              \end{array}
                            \right)
g\bigg)$, $x\in\mathbb{A}$,
\end{center}
which is continuous and satisfies $F(x+a)=F(x)$ for any $a\in \mathbb{Q}$. Thus it may be regarded as a function on the compact group $\mathbb{A}/\mathbb{Q}$. It therefore has a Fourier expansion in terms of the characters of $\mathbb{A}/\mathbb{Q}$. All the character has the form $x\mapsto \psi(\eta x)$, where $\eta \in \mathbb{Q}$.

Thus we have
$$
F(x)=\sum_{\eta \in \mathbb{Q}}F_{\eta}\psi(\eta x),$$
where $F_{\eta}=\int_{\mathbb{A}/\mathbb{Q}}F(y)\psi(-\eta y)dx=f_{\eta}=0$.
Then $F(x)=0$, hence $f(g)=F(0)=0$.
 \end{proof}

If $\varphi=\otimes_p\varphi_p \in S(V(\mathbb{A}))$ and $\varphi^{\prime}=\otimes_p\varphi_p^{\prime} \in S(V^{\prime}(\mathbb{A}))$  is a Kudla matching pair, i.e., $\lambda_V(\varphi)=\lambda_{V^{\prime}}(\varphi^{\prime})$, then $E(g, s, \varphi)=E(g, s, \varphi^{\prime})$. For $\tau=u+iv \in \H$, let $g_{\tau}=\kzxz{1}{u}{}{1}\kzxz{v^{\frac{1}{2}}}{}{}{v^{-\frac{1}{2}}}$, then one has:
\begin{proposition}\label{matching theorem}
Let notations be as above, one has
$$\tilde{I}(g_{\tau}, \varphi)=E(g_{\tau}, s, \varphi)=E(g_{\tau}, s, \varphi^{\prime})=I(g_{\tau},\varphi^{\prime} ),$$
where $\varphi_{\infty}=\varphi_{\infty}^{sp}$ and take $\varphi_{\infty}^{'}$ to be  $\varphi_{\infty}^{ra}$ or $\varphi_{\infty}^{sp}$ depends on $V^{\prime}$ is definite or indefinite,  $\varphi_{\infty}^{ra}$ and $\varphi_{\infty}^{sp}$ are defined in Section \ref{application}.  
\end{proposition}

\begin{proof}
Let $(\varphi, \varphi^{\prime})$ be a matching pair, and we know that \begin{equation}
E(g, s, \varphi)= E(g, s, \varphi^{\prime}),
\end{equation}
 for all $g \in SL_{2}(\mathbb{A})$.

Comparing the constant term $$E_P(g, s_0, \varphi^{\prime})= \Phi(g, s_0)+ \Phi_{1}(g, s_0)=\omega(g)\varphi^{\prime}(0)+ \Phi_1(g, s_0)$$ with $$I_P(g, \varphi^{\prime})=\omega(g)\varphi^{\prime}(0),$$ one has
 $ \Phi_1(g, s_0)=0$.
 
  From the Theorem \ref{result} and the Siegel-Weil formula, we obtain
 \begin{center}
 $\tilde{I}(g_{\tau}, \varphi)=E(g_{\tau}, s_{0}, \varphi)= E(g_{\tau}, s_{0}, \varphi^{\prime})= I(g_{\tau}, \varphi^{\prime})$.
 \end{center}
\end{proof}

\section{arithmetic geometry}\label{application}

Given an orthogonal decomposition
\begin{equation} \label{eq:spacedecomposition}
V = V^+ \oplus V^-,  \quad x = x^+ + x^-,
\end{equation}
with $V^+$ of signature $(2,0)$ and $V^-$ of signature $(0, 2)$. We define
\begin{equation}
\varphi^{sp}_\infty (x) = (4\pi (x^+, x^+) -1) e^{ -\pi (x^+, x^+) + \pi (x^-, x^-)}.
\end{equation}
 For any $x=\left(
                              \begin{array}{cc}
                                x_{1} & x_{2} \\
                                x_{3}&  x_{4}\\
                              \end{array}
                            \right) \in V$, let$$x^{+}=\left(
                              \begin{array}{cc}
                                \frac{x_{1}+x_{4}}2 & \frac{x_{2}-x_{3}}2 \\
                                \frac{x_{3}-x_{2}}2&  \frac{x_{1}+x_{4}}2\\
                              \end{array}
                            \right), x^{-}=\left(
                              \begin{array}{cc}
                                \frac{x_{1}-x_{4}}2 & \frac{x_{2}+x_{3}}2 \\
                                \frac{x_{3}+x_{2}}2&  \frac{x_{4}-x_{1}}2\\
                              \end{array}
                            \right)$$
                            and it is easy to check that this standard decomposition is exactly the orthogonal decomposition. Then one has
 \begin{equation}\label{equsplit}
 \varphi^{sp}_\infty (x)
=(2\pi ((x_{1}+x_{4})^{2}+(x_{2}-x_{3})^{2})-1)e^{-\pi(x_{1}^{2}+x_{2}^{2}+x_{3}^{2}+x_{4}^{2})}.
\end{equation}

Let $(V^{\prime}, Q^{'})$ be any other quaternion algebra over $\mathbb{Q}$, which is anisotropic. Let
\begin{equation}\label{Gaussian}
\varphi^{ra}_{\infty}(x) = e^{-2 \pi Q^{'}(x)}
\end{equation}
be the Gaussian of the space $V^{'}$.

 From \cite[Section 3.2]{DuYang} (Kudla used the notation $\tilde\phi(x, z)$), we know ($\varphi^{sp}_{\infty}(x)$, $\varphi^{ra}_{\infty}(x)$) is a matching pair such that
 \begin{equation}\lambda_V (\varphi^{sp}_{\infty})(g)=\lambda_{V^\prime}(\varphi^{ra}_{\infty})(g)=\Phi_\infty^2(g, 1),\end{equation} where $\Phi_\infty^{2}(gk_\theta,s)=e^{2\theta i}\Phi_\infty^{2}(g,s)$ for $k_\theta=\kzxz{cos\theta}{sin\theta}{-sin\theta}{cos\theta}$ and $\Phi_\infty^{2}(1,s)=1$.
 
\subsection{Definite quaternions and representation numbers}  \label{sect:definite}

We assume that $D>0$ has odd number of prime factors, and let $B=B(D)$ be the associated definite quaternion. In this case,  $V^\prime= (B, \det)$ is positive definite. 

 For any $\varphi_f^{'} \in S(V(\A_f))$,   the theta kernel
$$
\theta(\tau, h, \varphi_f^{'} \otimes\varphi^{ra}_{\infty}) = v^{-1} \theta(g_\tau, h, \varphi_f^{'}\otimes \varphi^{ra}_{\infty})
$$
is a holomorphic modular form of weight $2$ for some congruence subgroup. Here $g_\tau = n(u) m(\sqrt v)$ for $\tau =u + i v \in \mathbb H$. So the integral
$$
I(\tau, \varphi_f^{'}\otimes\varphi^{ra}_{\infty}): = v^{-1} I(g_\tau, \varphi_f^{'}\otimes\varphi^{ra}_{\infty})
$$
is also a modular form of weight $2$.

For an even integral lattice $L$ in $V^{'}$, we  let
\begin{equation} \label{eq:new4.1}
\theta(\tau, L) = \theta(\tau,  \cha(\widehat L) \otimes\varphi^{ra}_{\infty}), \quad I(\tau, L)=I(\tau,  \cha(\widehat L)\otimes \varphi^{ra}_{\infty}),
\end{equation}
where $\widehat L=L\otimes \widehat{\Z}$.
Two lattices $L_1$ and $L_2$  in $V^{\prime}$ are equivalent if there is $h \in O(V^{'})(\Q)$ such that $hL_1 =L_2$. Two lattices $L_1$ and $L_2$  are in the same genus if they are equivalent locally everywhere, i.e,  there is $h \in O(V^{'}) (\A)$ such that $h L_1=  L_2$.  Notice that $O(V^{'})(\A)$ acts on the set of lattices as follows: $h L =  (h_f \widehat L)\cap V^{\prime}$ where $h_f$ is the finite part of $h=h_f h_\infty$.  Let $\gen(L)$ be the genus of $L$---the set of equivalence classes of lattices in the same genus of $L$. Then one has 
$$
O(V^{'})(\Q) \backslash O(V^{'})(\A)/K(L) O(V^{'})(\R) \cong  \gen(L), \quad [h] \mapsto hL,
$$
where $K(L)$ is the stabilizer of $\widehat L$ in $O(V^{'})(\A)$.

\begin{lemma} \cite[Proposition 4.1]{DuYang}\label{Duyanglemma}
Let
$$
r_L(n) =|\{ x \in  L:\,  Q(x) =m\}|,$$
$$\quad  r_{\gen(L)}(m) = \left(\sum_{L' \in \gen(L)} \frac{1}{|O(L')|}\right)^{-1} \sum_{L' \in \gen(L)}  \frac{r_{L'}(m)}{|O(L')|},
$$
where $O(L^{'})$ is the stabilizer of $L^{'}$ in $O(V^{'})$. Then one has
\begin{align*}
I(\tau, L)=& \sum_{m=0}^\infty r_{\gen(L)}(m) q^m,
\end{align*}
where $q =e(\tau)$.
\end{lemma}
 In this paper, we choose $L=(\OO_D(N), \det)$ as
 an even integral lattice in $V^{\prime}$, then $r_L(m)=r_{D,N}(m)$.
\subsection{Indefinite quaternions and Shimura curves}

In this subsection, we assume that $D>0$ has even number of prime factors, then $B=B(D)$ is an indefinite quaternion. In this case, $V^\prime= (B, \det)$ is of signature $(2, 2)$. When $D>1$, according to \cite[Theorem 4.23]{KuIntegral},  the theta integral $I(g, \varphi)$ is a generating function of degrees of  some divisors with respect to the tautological line bundle over the Shimura curve associated to $V^\prime$. In our case, the divisors can be identified with Hecke correspondences on a Shimura curve. When $D=1$, the case is similar.

 For a positive integer $m$, let $T_{D, N}(m)$ be the Hecke correspondence on Shimura curve $X_0^D(N)=\Gamma_0^D(N) \backslash  \mathfrak{H}$ which is defined in Subsection \ref{quaternion}.
 The normalized degree is defined by
\begin{equation}
r_{D, N}(m)= -\frac{2}{ \vol(X_0^D(N), \Omega_0)}\deg T_{D,  N}(m).
\end{equation}
We have the following result.


\begin{lemma} \label{DuYang2} For $\varphi_f =\cha(\widehat{\OO_D(N)})$, when $D>1$ (\cite[Theorem 5.3]{DuYang}) one has
$$
I(\tau, \varphi_f \otimes\varphi^{sp}_{\infty} ) = v^{-1} I(g_\tau, \varphi_f\otimes\varphi^{sp}_{\infty} ) = \sum_{m=0}^\infty r_{D, N}(m) q^m,
$$
and when $D=1$,
$$
\tilde{I}(\tau, \varphi_f \otimes\varphi^{sp}_{\infty} ) = v^{-1} I(g_\tau, \varphi_f\otimes\varphi^{sp}_{\infty} ) = \sum_{m=0}^\infty r_{D, N}(m) q^m,
$$
where  $r_{D, N}(0) =1$.
\end{lemma}

\begin{proof}
The case $D=1$ follows from the same proof as in  \cite[Theorem 5.3]{DuYang}, and we leave it to the readers.  
\end{proof}

\subsection{Relations to other quaternions}

In this subsection, let $V^{\prime}$ be any division quaternion algebra over $\mathbb{Q}$ of discriminant $D>1$ over $\Q$. We prove Theorem \ref{rDN} as follows which could be used to compute the number $r_{D, N}$.

\begin{proof}[\bf Proof of Theorem \ref{rDN}]
Recall the definition in Section \ref{sect:Preli},
$$\varphi^{\prime}=\begin{cases}
\cha(\widehat{\mathcal{O}_D(N)}) \otimes\varphi^{ra}_{\infty} \in S(V^{\prime}(\mathbb{A})) & \ff V^{\prime}~is~definite,\\
\cha(\widehat{\mathcal{O}_D(N)}) \otimes\varphi^{sp}_{\infty} \in S(V^{\prime}(\mathbb{A}))  & \ff V^{\prime}~is~indefinite.
\end{cases}$$
By Proposition \ref{globalmatching}, we know $(\varphi^{\prime},\varphi_{D}^{N} )$ is a matching pair, where $\varphi_{D}^{N}=\otimes_{p\leq \infty}\varphi_{p}  \in S(V(\mathbb{A}))$ is constructed in this proposition. From Kudla's matching and the Siegel-Weil formula, one has 
$$I(g_{\tau}, \varphi^{\prime})=E(g_\tau, s_0, \varphi_{D}^{N}).$$   
Comparing the Fourier coefficients of both sides, we have
\begin{eqnarray}
&&r_{D,N}(m)=q^{-m}v^{-1}E_m(g_\tau, s_0, \varphi_{D}^{N})\\
&=& v^{-1}q^{-m}\prod_{p< \infty}W_{m, p}(e, 1, \Phi_{p})\times W_{m, \infty}(g_{\tau}, 1, \Phi_{\infty}),\nonumber
\end{eqnarray}
where $\Phi_{p}(g_p, 1)=\lambda_{p}(\varphi_{p})(g_p)$.
Here the local Whittaker function $$W_{m, p}(g, 1, \Phi_{p})=\int_{\mathbb{Q}_{p}} \Phi_{p}( wn(b)g, 1) \psi_{p} (-m b)db$$ is defined by equation (\ref{whittakerfun}).

When $p=\infty$, it is known by \cite[Proposition 15.1]{KRYComp}
\begin{equation}\label{infty}
W_{m, \infty}(g_{\tau}, 1, \Phi_{\infty})=-4 \pi^{2}m q^m v.
\end{equation}
When $p$ is finite, $W_{m, p}(e, 1, \Phi_{p})$ 
is given in the Lemma \ref{densityfinite}. Then one has
\begin{eqnarray}
&&r_{D,N}(m)=-4 \pi^{2}m \prod_{p< \infty}W_{m, p}(e, 1, \Phi_{p})\\
&=&(-1)^{k+1}24m\prod_{p \nmid ND}\frac{p-p^{-\ord_pm}}{p-1}\prod_{p \mid D}\frac{1}{(p-1)p^{\ord_pm}}\prod_{p \mid N}\nonumber\\
&&\times \frac{2p-p^{-(\ord_pm-1)}-p^{-\ord_pm}}{p^{2}-1}.\nonumber
\end{eqnarray}
\end{proof}
\begin{lemma}\label{densityfinite}
Assume $\varphi_{D}^{N}=\otimes_{p\leq \infty}\varphi_{p}  \in S(V(\mathbb{A}))$, and $\Phi_{p}(g_p, 1)=\lambda_{p}(\varphi_{p})(g_p)$, then 
\begin{equation}
W_{m, p}(e, 1, \Phi_{p})
= \begin{cases}
(1-p^{-2})\sum_{i=0}^{\ord_pm}p^{-1} & \ff p \nmid ND\nonumber\\
2p^{-1}-p^{-\ord_pm-1}-p^{-\ord_pm-2} & \ff \nonumber p \mid N \\
-p^{-\ord_pm-2}(p+1) & \ff p \mid D.\nonumber\\
\end{cases}
\end{equation}
\end{lemma}
\begin{proof}
It suffices to compute the integral
 $\int_{\mathbb{Q}_{p}} \Phi_{i}^{sp}( w n(b), 1) \psi_{p} (-m b)db$, where
 $\Phi_{i}^{sp}(g_p, 1)=\lambda_{p}(\varphi_{i}^{sp})(g_p)$, $i=0, 1$. 
Here $\varphi_{i}^{sp}$ are defined in Proposition \ref{globalmatching}.

One has the following decomposition 
$$\SL_2(\Z_p)=K_0(p) \cup N(\Z_p) w K_0(p),$$
where $K_0(p)=\{\kzxz{a}{b}{c}{d} \mid a,b,d\in \Z_p, c \in p\Z_p\}$.
 It is easy to check that $\Phi_{i}^{sp}(g_p, 1) \in I(1,\chi)^{K_0(p) }$, which is the $K_0(p)$-invariant subspace in the induced representation $I(1,\chi)$. Here $\chi$ is trivial. So any $\Phi \in I(1,\chi)^{K_0(p) }$ is determined by $\Phi(e, 1)$ and $\Phi(w, 1)$. It is easy to know that 
 \begin{equation}\Phi_{0}^{sp}(e, 1)=\Phi_{0}^{sp}(w, 1)=\Phi_{1}^{sp}(e, 1)=1,~\Phi_{1}^{sp}(w, 1)=\frac{1}{p},
 \end{equation}
and 
$$
\int_{\mathbb{\Z}_{p}^{\times}}\psi_{p}(ab)db=\begin{cases}
1-p^{-1} & \ff \ord_p(a) \geq 0,\\
-\frac{1}{p} & \ff \ord_p(a) = -1,\\
0 & \ff \ord_p(a) \leq -1.\\
\end{cases}
$$
It is known that 
 $w n(b)=\kzxz{-b^{-1}}{1}{}{-b} n_{b^{-1}}$ with $n_{b^{-1}}=\kzxz{1}{}{b^{-1}}{1}$,
  and $$\Phi( w n(b), 1)=\mid b \mid ^{-2}\Phi(n_{b^{-1}} , 1),
~for ~any~\Phi \in I(1,\chi).$$
 
Then one has
 \begin{eqnarray}\label{densityodd}
 &&~W_{m, p}(e, 1, \Phi_{0}^{sp})\\
 &=&  \int_{\mathbb{\Z}_{p}} \Phi_{0}^{sp}( w n(b), 1) \psi_{p} (-m b)db+
 \sum_{k>0} \int_{p^{-k}  \mathbb{\Z}_{p}^{\times}}\Phi_{0}^{sp}( w n(b), 1) \psi_{p} (-m b)db\nonumber\\
 &=&1+\sum_{k>0}p^{-2k}\int_{p^{-k}  \mathbb{\Z}_{p}^{\times} }\Phi_{0}^{sp}(  n_{b^{-1}}, 1) \psi_{p} (-m b)db\nonumber\\
 &=&1+\sum_{k>0}p^{-k}\int_{  \mathbb{\Z}_{p}^{\times} } \psi_{p} (-p^{-k}m b)db\nonumber\\
 &=&(1-p^{-2})\sum_{i=0}^{\ord_pm}p^{-i}.\nonumber
 \end{eqnarray}

By the same method, one has 
\begin{eqnarray}\label{densitysplit}
 &&~W_{m, p}(e, 1, \Phi_{1}^{sp})\\
 &=&  \int_{\mathbb{\Z}_{p}} \Phi_{1}^{sp}( w n(b), 1) \psi_{p} (-m b)db+
 \sum_{k>0} \int_{p^{-k}  \mathbb{\Z}_{p}^{\times}}\Phi_{1}^{sp}( w n(b), 1) \psi_{p} (-m b)db\nonumber\\
 &=&p^{-1}+\sum_{k>0}p^{-2k}\int_{p^{-k}  \mathbb{\Z}_{p}^{\times} }\Phi_{1}^{sp}(  n_{b^{-1}}, 1) \psi_{p} (-m b)db\nonumber\\
 &=&p^{-1}+\sum_{k>0}p^{-k}\int_{  \mathbb{\Z}_{p}^{\times} } \psi_{p} (-p^{-k}m b)db\nonumber\\
 &=&2p^{-1}-p^{-(\ord_pm+1)}-p^{-(\ord_pm+2)}.\nonumber
 \end{eqnarray}
Combining it with equation (\ref{densityodd}) and the Proposition \ref{globalmatching}, one could obtain the result.

\end{proof}
\begin{proof}[\bf{Proof of Corollary \ref{degree}}]
By  \cite[(2.7)]{KRYComp} and
\cite[Lemma 5.3.2]{Miy}, one has
\begin{align} \label{eq:volume}
\vol(X_0^D(N), \Omega_0) &:= \int_{X_0^D(N)} \Omega_0 = -2 [\OO_D^1 : \Gamma_0^D(N)] \zeta_D(-1)
\\
   &=\frac{ DN}6 \prod_{p|N} (1+p^{-1}) \prod_{p|D} (1-p^{-1} ) .\nonumber
\end{align}
Then from Theorem \ref{rDN}, we get the result.
\end{proof}

\section{Three and Four squares Problem}\label{foursquare}

In this section, we prove the Theorem \ref{threefoursquareth}. We split it into two cases:
one is four squares, the other is three squares .
 
 \subsection{Four squares sum} \label{secfour}
 The quaternion algebra associated to quadratic form $Q=x_1^2+x_2^2+x_3^2+x_4^2$ is $B(2)$, which could be written as 
 $$B(2)=\Q +\Q i + \Q j+\Q k, i^2=j^2=k^2=-1 ~and~ ij=-ji=k.$$
  We denote $V^{\prime}=(B(2), Q)$ and consider the order $ \mathcal{O}=\Z+\Z i+\Z j+\Z k$, which is not the maximal order or Eichler order. 
   Now we fix the lattice $L=(\mathcal{O}, Q)$, then one has 
 \begin{equation}
 r_{4}(m)=r_{L}(m), ~m>0.
 \end{equation}
  The class of $L$ in $gen(L)$ is 1, so we could compute it by Eisenstein series.
  Now let $\varphi_{p}^{\prime}=\cha(L_p)$ and $\varphi_{\infty}^{\prime}=\varphi_{\infty}^{ra}$, and denote $\varphi^{\prime}=\otimes_{p\leq \infty}\varphi_{p}^{\prime} \in S(V^{\prime}(\A))$.  Here $L_p=L\otimes\Z_p$.
  \begin{theorem}
  With the above notations, one has 
  $$r_{4}(m)=\sum_{d \mid m, 4\nmid d}d.$$
  \end{theorem}
  \begin{proof}
  From the Siegel-Weil formula and Lemma \ref{Duyanglemma}, one has 
  \begin{equation}\label{equfoursquare}
  E(\tau, 1, \varphi^{\prime})=I(\tau, \varphi^{\prime})=\sum_{m\geq 0}r_{4}(m)q^m.
  \end{equation}
It's an Eisenstein series of weight $2$. 
 From above equation,
one has
 \begin{eqnarray}
r_{4}(m)&=&E_m(\tau, 1, \varphi^{\prime})q^{-m}\nonumber\\
&=&v^{-1}q^{-m}\prod_{p< \infty}W_{m, p}(e, 1, \Phi_{p})\times W_{m, \infty}(g_{\tau}, 1, \Phi_{\infty}),\nonumber
\end{eqnarray}
where $\Phi_p=\lambda(\varphi_{p}^{\prime})$.

By the equation (\ref{infty}), we have
\begin{equation}\label{whitinfinite}
W_{m, \infty}(g_{\tau}, 1, \Phi_{\infty})=-4 \pi^{2}m q^m v.
\end{equation}

When $p$ is odd, one has $L_{p}^{\sharp}=L_{p}$, and the dual lattice is given by $L_{p}^{\sharp}=\{x \in  V^{\prime}(\Q_p)\mid (x, L_p)\subseteq \Z_{p}\}$. Then $\Phi_{p}(g, s)$ is spherical, i.e., $\SL_{2}(\Z_p)$-invariant and $\Phi_{p}(e, s)=1$.
By the equation (\ref{densityodd}), one has
\begin{equation}\label{spherical}
W_{m, p}(e, 1, \Phi_{p})=(1-p^{-2})\sum_{i=0}^{\ord_pm}p^{-i}=\frac{\sigma_{-1,p}(m)}{\zeta_p(2)},
\end{equation}
where $\sigma_{-1,p}(m)=\sum_{i=0}^{\ord_pm}p^{-i}.$

 The case $p=2$ is different. 
 Notice that $L_{2}^{\sharp}/L_{2}\cong (\Z/2\Z)^4$ with $L_{2}^{\sharp}=\frac{1}{2}L_{2}$. Let $K_0(4)=\{\kzxz{a}{b}{c}{d} \mid a,b,d\in \Z_2, c \in 4\Z_2\}$, then one has the following decomposition 
\begin{equation}\SL_2(\Z_2)=K_0(4) \cup n_2 K_0(4) \cup N(\Z_2) w K_0(4),~ ~n_2=\kzxz{1}{}{2}{1}.
\end{equation}
 It is easy to see that $\Phi_{2}(g_2, 1) \in I(1,\chi)^{K_0(4) }$, which is the $K_0(4)$-invariant subspace in  $I(1,\chi)$.  
One need to check that $\varphi_{2}^{\prime}$ is $K_0(4)$-invariant for the  Weil representation. We check
 $\omega(n_4)\varphi_2^{\prime} =\varphi_2^{\prime}$ and leave others to the reader, where $n_4= w^{-1} n(-4) w=\kzxz{1}{}{4}{1}$. 
 
 One has
$$
\omega(w)\varphi_{2}^{\prime}(x) =\gamma(V^{\prime}_2) \vol(L_2)\cha(L_{2}^{\sharp})(x) =-\frac{1}{4}\cha(L_{2}^{\sharp})(x).
$$
So
$$
\omega(n(-4)w)\varphi_{2}^{\prime}(x)=-\frac{1}{4} \psi_2(-4 Q (x)) \cha(L_{2}^{\sharp})(x)=-\frac{1}{4}\cha(L_{2}^{\sharp})(x),
$$
i.e.,
$$
\omega(n(-4)w)\varphi_{2}^{\prime}= \omega(w)\varphi_{2}^{\prime}.
$$
Then we obtain
$$
\omega(n_4) \varphi_2^{\prime} = \omega( w^{-1}) \omega(n(-4)w)\varphi_2^{\prime} =\varphi_2^{\prime}.
$$

 From the Weil representation, one has 
 $$\Phi_{2}(e, 1)=1, \Phi_{2}(n_2, 1)=0, \Phi_{2}(w, 1)=-\frac{1}{4}.$$ 
By the same methods as equation (\ref{densityodd}), we know
 \begin{eqnarray}\label{densitytwo}
 &&W_{m, 2}(e, 1, \Phi_{2})\\
 &=&  \int_{\mathbb{\Z}_{2}} \Phi_{2}( w n(b), 1) \psi_{2} (-m b)db+
 \sum_{k>0} \int_{2^{-k} U }\Phi_{2}( w n(b), 1) \psi_{2} (-m b)db\nonumber\\
 &=&-\frac{1}{4}+\sum_{k>0}2^{-2k}\int_{p^{-k} U }\Phi_{2}(  n_{b^{-1}}, 1) \psi_{2} (-m b)db,\nonumber
 \end{eqnarray}
 where $U=1+2\Z_2$ and $ n_{b^{-1}}=\kzxz{1}{}{b^{-1}}{1}$.

Now we split it into two cases: $m$ is odd or even.

{\bf When $m$ is odd }

From the equation (\ref{densitytwo}), one has
\begin{eqnarray}
&& W_{m, 2}(e, 1, \Phi_{2})\nonumber\\
&=& -\frac{1}{4}+2^{-2}\int_{2^{-1} U }\Phi_{2}(  n_{b^{-1}}, 1) \psi_{2} (-m b)db\nonumber\\
&&+\sum_{k>1} 2^{-2k}\int_{2^{-k} U }\Phi_{2}( n_{b^{-1}}, 1) \psi_{2} (-m b)db\nonumber\\
&=&-\frac{1}{4}.
\end{eqnarray}
Notice that 
$$
\int_{U}\psi_{2}(ab)db=\begin{cases}
\frac{1}{2} & \ff \ord_2(a) \geq 0,\\
-\frac{1}{2} & \ff \ord_2(a) = -1,\\
0 & \ff \ord_2(a) \leq -1.\\
\end{cases}
$$
Combing it with equations (\ref{infty}) and  (\ref{spherical}), we obtain
\begin{eqnarray}\label{oddcase}
r_{4}(m)=8\sigma_{1}(m),
\end{eqnarray}
where $\sigma_{1}(m)=\sum_{d \mid m}d$.

{\bf When $m$ is even }  

Assuming $\ord_{2}(m)=r$, we have
\begin{eqnarray}
W_{m, 2}(e, 1, \Phi_{2})
&=& -\frac{1}{4}+\sum_{k>1}^{r+1}2^{-2k}\int_{2^{-k} U }\Phi_{2}(  n_{b^{-1}}, 1) \psi_{2} (-m b)db\nonumber\\
&&+\sum_{k>r+1} 2^{-2k}\int_{2^{-k} U }\Phi_{2}( n_{b^{-1}}, 1) \psi_{2} (-m b)db\nonumber\\
&=&-\frac{1}{4}+2^{-3}+\cdots 2^{-r-1}-2^{-r-2}\nonumber\\
&=&-2^{-r-1}-2^{-r-2}.\nonumber
\end{eqnarray}
Combing it with equations (\ref{whitinfinite}) and  (\ref{spherical}), one obtains
\begin{eqnarray}
r_{4}(m)&=&(2^{-r-1}+2^{-r-2})4 \pi^{2}m \prod_{ p\neq 2}\frac{\sigma_{-1,p}(m)}{\zeta_p(2)}\nonumber\\
&=& 24 \sigma_{1}(\frac{m}{2^r})=8 \sigma_{1}(2)\sigma_{1}(\frac{m}{2^r})=8 \sigma_{1}(\frac{m}{2^{r-1}})=8 \sum_{d \mid m, 4 \nmid d}d.
\end{eqnarray}
Thus we finish the proof.
  \end{proof}

{\bf The second proof:}

By the Kudla's matching, we could compute $r_4(m)$ via the quadratic space $V=(M_2(\Q), Q)$.
We choose $\varphi=\otimes_p\varphi_p \in S(V(\A))$ as follows, 
$$\varphi_p=
\begin{cases}
\varphi^{sp}_\infty & \ff p=\infty,\\
 -\frac{1}{2}\varphi_0^{sp} -\frac{1}{2}\varphi_1^{sp}+2\varphi_2^{sp} & \ff p=2,\\
\varphi_{0}^{sp} &  otherwise,
\end{cases}$$ which satisfies $\lambda_p(\varphi_p)=\lambda_p(\varphi^{\prime}_p)=\Phi_p$.
Then one has 
\begin{equation}\label{Kudlamatch}
I(g_{\tau}, \varphi^{\prime})=E(g_{\tau}, 1, \varphi^{\prime})=E(g_{\tau}, 1, \varphi).
\end{equation}

Comparing $m$-th Fourier coefficient of above equation, one has
\begin{eqnarray}
&&r_{4}(m)=v^{-1}q^{-m}\prod_{p<\infty}W_{m, p}(e, 1, \varphi_{p}) W_{m, \infty}(g_{\tau}, 1, \varphi_{\infty}).
\end{eqnarray}

When $p=2$, $\Phi_{p}=-\frac{1}{2}\Phi_{0}^{sp} -\frac{1}{2}\Phi_{1}^{sp}+2\Phi_{2}^{sp}$, where $\Phi_{i}^{sp}=\lambda_p(\varphi_i^{sp})$  and $\varphi_i^{sp}$ is  defined in Section \ref{sect:Preli}. 

When $i=0, 1$, the integral $W_{m, p}(g_p, 1, \Phi_{i}^{sp}) $ are given in the Lemma \ref{densityfinite}, and we leave $W_{m, p}(g_p, 1, \Phi_{2}^{sp})$ to the readers.  When $p\neq 2$, the local integral can be computed similarly as the above theorem.
Then one obtains the expected result.

\subsection{ Three squares sum}
\label{secthree}
In this subsection, we consider the space $B^0(2)=\{ x \in B(2); \tr(x)=0\}$ and denote the quadratic space $V=(B^0(2), Q)$. Fix the lattice $\mathcal{L}=\Z i+\Z j+\Z k$. This lattice has genus one, thus one has 
$$r_{3}(m)=r_{\mathcal{L}}(m).$$

Define
\begin{equation}
b_p=\begin{cases}
\frac{1-p^{-1}\chi_{d}(p)+p^{-t_p-1}\chi_{d}(p)-p^{-t_p-1}}{1-p^{-1}} &\ff p ~is ~ odd,\\
\frac{(1-\chi_{d}(p))(p^{-t_p}-p^{-t_p-2})}{1-p^{-1}}& \ff p=2.
\end{cases}
\end{equation}

Let $\widetilde{\SL_2}$ be the metaplectic double cover of $\SL_2$, and we could write it as $\widetilde{\SL_2}=\SL_2\times \{\pm 1\} $. The multiplication is given by 
$$[g_1, \varepsilon_1][g_2, \varepsilon_2]=[g_1g_2, \varepsilon_1\varepsilon_2c(g_1, g_2)],$$
for the cocycle as in \cite{KRYannal}. We frequently abuse the notation and write $n(b)$, $m(a)$, $w$ for the elements $[n(b, 1]$, $[m(a)]$ and $[w, 1]$.

Set $\varphi=\varphi_{f} \otimes\varphi_{\infty} \in S(V(\A))$, where $\varphi_{f}=char(\widehat{\mathcal{L}})$  and $\varphi_{\infty}=e^{-2\pi Q(x)}$.
For $m>0$, let $-4m=dc^2$, where $d$ is the fundamental discriminant. Let $\chi_d$ be the corresponding  Dirichlet character of field $\Q(\sqrt{d})$ 
and
\begin{equation}
w=\begin{cases}
2 & \ff d<-4,\\
4 &   \ff d=-4,\\
6 & \ff d=-3.
\end{cases}
\end{equation}
\begin{theorem}\label{threeth}
With the above notations, one has
\begin{equation}\label{threesum}
r_3(m)=\frac{24 h(d) }{w}(1-\chi_d(2))\sum_{l \mid c, (l, 2)=1}l \prod_{q \mid l}(1-\chi_{d}(q)q^{-1}),
\end{equation}
where $h(d)$ is the class number of $\Q(\sqrt{d})$.
\end{theorem}
\begin{proof}
By the Siegel-Weil formula and Lemma \ref{Duyanglemma}, we have 
  \begin{equation}\label{equfoursquare}
  E(\tau, \frac{1}{2}, \varphi)=I(\tau, \varphi)=\sum_{m\geq 0}r_{3}(m)q^m.
  \end{equation}
Comparing Fourier coefficients of both sides, one has
\begin{eqnarray}\label{threeprod}
&&r_{3}(m)=E_m(\tau, \frac{1}{2}, \varphi)q^{-m}\nonumber\\
&=&v^{-\frac{3}{4}}q^{-m} \prod_{p< \infty}W_{m, p}(e, \frac{1}{2}, \Phi_{p})\times W_{m, \infty}(g_{\tau}, \frac{1}{2}, \Phi_{\infty}),
\nonumber
\end{eqnarray}
where $\Phi_p=\lambda(\varphi_{p})$. We identify notation $g_{\tau}$ with $[g_{\tau}, 1] \in \widetilde{\SL_2}(\R)$. Following the methods given in \cite{KY}, we compute the local Whittaker functions as follows.
It is known that
\begin{equation}\label{threeinfinite}
 W_{m, \infty}(g_{\tau}, \frac{1}{2}, \Phi_{\infty})=-4 \pi\sqrt{2m}\zeta_8v^{\frac{3}{4}} q^m.
\end{equation}
Denote $t_p=\ord_p(c)$.

 When p is odd,
one has
\begin{eqnarray}\label{threeodd}
&&W_{m, p}(e, \frac{1}{2}, \Phi_{p})\\
&=&\begin{cases}
 1+p^{-1}-p^{-t_p-1}+p^{-t_p-1}\chi_{d}(p) &\ff p \nmid d,\\
 1+p^{-1}-p^{-t_p-1}-p^{-t_p-2} & \ff p \mid d,
 \end{cases}\nonumber\\
 &=&\frac{L_p(1, \chi_{d})}{\zeta_p(2)}b_p,\nonumber
\end{eqnarray}
where \begin{eqnarray}
b_p&=&\frac{1-p^{-1}\chi_{d}(p)+p^{-t_p-1}\chi_{d}(p)-p^{-t_p-1}}{1-p^{-1}}\nonumber\\
&=&p^{-t_p}\sum_{l \mid p^{t_p}}l \prod_{q \mid l}(1-\chi_{d}(q)q^{-1}).
\end{eqnarray}

When $p=2$,
\begin{equation}\label{threetwo}
W_{m, 2}(e, \frac{1}{2}, \Phi_{2})=-\frac{\zeta_8^{-1}}{2\sqrt{2}}L_{2}(1, \chi_{d})b_{2},\end{equation}
where
\begin{eqnarray}
b_{2}=\frac{(1-\chi_{d}(2))(2^{-t_2}-2^{-t_2-2})}{1-2^{-1}}=\frac{3}{2}2^{-t_2}(1-\chi_d(2)).
\end{eqnarray}
So we obtain
\begin{equation}
\prod_{p}b_p=\frac{3}{2}c^{-1}(1-\chi_d(2))\sum_{l \mid c, (l, 2)=1}l \prod_{q \mid l}(1-\chi_{d}(q)q^{-1}),
\end{equation}
where $q$ runs over prime factors of $l$.
By this equation and equations  (\ref{threeinfinite}), (\ref{threeodd}) and (\ref{threetwo}),  one has 
\begin{eqnarray}\label{threefinal}
r_3(m)&=& 2 \pi\sqrt{m}L_{2}(1, \chi_{d})\prod_{p~odd}\frac{L_p(1, \chi_{d})}{\zeta_p(2)} \prod_{p }b_p \nonumber\\
&=&\frac{24 h(d) }{w}(1-\chi_d(2))\sum_{l \mid c, (l, 2)=1}l \prod_{q \mid l}(1-\chi_{d}(q)q^{-1}).
\end{eqnarray}

\end{proof}

\begin{remark}
From equation (\ref{threefinal}), it is easy to see that $r_{3}(m)=r_{3}(m)$.
If $m\equiv 7 (\rm mod 8)$, then $d \equiv 5 (\rm mod 8)$.
So one knows that $\Q(\sqrt{d})$ is split at $2$, thus $a_2=0$. We recover the well known result that $r_3(m)=0$ when $m=4^{a}(8k+7)$, for $a, k>0$.
\end{remark}
For a discriminant $m >0$, define Hurwitz class number
$H(m)$ be the number of classes of not necessarily primitive positive definite quadratic forms of discriminant $-m$, except that those classes which have a representative which is a multiple of the form$x^2+y^2$ should be counted with weight $1/2$ and those which have a representative which is a multiple of the form $x^2 + xy + y^2$ should be counted with weight $1/3$. The number $H(m)$ is given by
\begin{equation}
H(m)=\frac{2h(d)}{w}\sum_{l \mid f}l \prod_{p \mid l}(1-\chi_{d}(p)p^{-1}),
\end{equation}
where $-m=df^2$. When $m$ is not a discriminant, $H(m)=0$.

Comparing this equation with equation (\ref{threefinal}), one obtains the following result easily.
\begin{corollary}[HZ, Section 2.1]\label{HZcor}
Let $m>0$, one has
\begin{equation}
r_3(m)=
\begin{cases}
12 H(4m) &\ff  m\equiv 1~ or~2 ~({\rm mod} 4),\\
24 H(m) &\ff  m\equiv 3~ (\rm mod 8),\\
0 &\ff  m\equiv 7~(\rm mod 8),\\
 r_3(\frac{m}{4}) &\ff  m\equiv 0 ~(\rm mod 4).
\end{cases}
\end{equation}
\end{corollary}
Since $\sum_{m\geq 0}r_{3}(m)q^m$ is an Eisenstein series, then one could ask a question how about $H(\tau)=\sum_{m\geq 0}H(m)q^m$, where $H(0)=-\frac{1}{12}$ and $q=e^{2\pi i \tau}$. 

 Zagier discovered the well known Eisenstein series in \cite[Theorem 2, Chapter 2]{HZ}.
\begin{equation}
E(\tau)= H(\tau)+\frac{1}{16 \pi \sqrt{v}}\sum_{m=-\infty}^{\infty}\beta(4 \pi m^2 v)q^{-m^2},
\end{equation}
where $v=\Im(\tau)$ and 
$$\beta(t)=\int_{1}^{\infty}u^{-\frac{3}{2}}e^{-tu}du,~t\geq0.$$
$E(\tau)$ is a nonholomorphic Eisenstein series of weight $3/2$ and $H(\tau)$ is the holomorphic part.

\section{Hardy's singular series}\label{sechardy}

 Hardy defined the singular series by
\begin{equation}
\rho_{s}(m)=\frac{\pi^{\frac{s}{2}}}{\Gamma(\frac{s}{2})}m^{\frac{s}{2}-1} \mathfrak{G}_{s}(m),
\end{equation} 
where 
\begin{equation}
\mathfrak{G}_{s}(m)=\sum_{k=1}^{\infty} A_k(m),
\end{equation}
and 
\begin{equation} A_k(m)=\sum_{h=1, (h,k)=1}^{k}(\frac{1}{k}\sum_{j=1}^ke^{2\pi i hj^2/k})^{s}e^{-2\pi imh/k}, ~A_1(m)=1
\end{equation}

It is known by Hardy \cite{Ha2} that
\begin{equation}
A_{k_1k_2}(m)=A_{k_1}(m)A_{k_2}(m), (k_1, k_2)=1.
\end{equation}
Thus it suffices to study $A_k$ for $k$ a prime power.
We write
\begin{equation}
\mathfrak{G}_{s}(m)=\sum_{k=1}^{\infty} A_k(m)=\prod S_p(m),
\end{equation}
where $$S_{p}(m)=\sum_{r=0}^{\infty}A_{p^r}(m).$$

Now we fix $s=3$ and denote $r_p=\ord_p(m)$. 
By Hardy's formula, and also in Dickson\cite{Di}, one has the following result:\\
If $r$ is an odd positive integer,
\begin{eqnarray}
&&A_{2^r}(m)=
\begin{cases}
0 & \ff r>r_p+3,\\
2^{-\frac{r-1}{2}} \cos\frac{2^{-r+3}m -3}{4}\pi 
& \ff r\leq r_p+3, 2^{-r+3}m \equiv 3~ \rm (mod ~4),\\
0 & \ff r\leq r_p+3, 2^{-r+3}m \not\equiv 3~\rm (mod ~4),
\end{cases}\nonumber\\
&&A_{p^r}(m)=
\begin{cases}
0 & \ff r>r_p +1,\\
-p^{-\frac{r+1}{2}} (\frac{-p^{-r+1}m}{p})& \ff r= r_p +1,\\
0 & \ff r< r_p +1.
\end{cases}\nonumber
\end{eqnarray}
If $r$ is an even positive integer,
\begin{eqnarray}
&&A_{2^r}(m)=
\begin{cases}
0 & \ff r>r_p+2,\\
2^{-\frac{r-1}{2}} \cos\frac{2^{-r+3}m -3}{4}\pi 
& \ff r\leq r_p+2,\\
\end{cases}\nonumber\\
&&A_{p^r}(m)=
\begin{cases}
0 & \ff r>r_p +1,\\
-p^{-\frac{r}{2}-1} & \ff r= r_p +1,\\
(p-1)p^{-\frac{r}{2}-1} & \ff r< r_p +1.
\end{cases}\nonumber
\end{eqnarray}

Now we recall the normalized local Whittaker functions. For details see  \cite[Section 4]{KY}.

Let $(V, \Q)$ be a quadratic space with dimension $n$. and let $L$ be a lattice in $V$.
For any $\varphi_p \in S(V_{p})$, the Whittakler function 
$$W_{m, p}(e, s_{0}, \varphi_{p})=\gamma(V_p)\int_{\Q_p}\int_{V_p}\psi_{p}(bQ(x))\varphi_{p}(x)d_{V}x\psi_{p}(-mb)db,$$
where $d_{V}x$ is the self dual measure with respect to $\psi_p((x, y))$ and $\gamma(V_p)$ is the Weil index.

Fix a $\Z_p$-basis $\{e_i \}$ of $L_p$, then $L_p\cong \Z_p^n$ . Let $dx=\prod dx_i$ be the standard 
measure on $V$, then one has 
\begin{equation}
d_Vx=[L^{\sharp}_p : L_p]^{-\frac{1}{2}}dx= \mid \det G  \mid_{p} ^{\frac{1}{2}},
\end{equation}
where $G$ is the Gram matrix $((e_i, e_j))$ for $L_p$.
Kudla and Yang defined 
\begin{equation}\label{norwhitt}
W_{p}( s_{0}, m)=\int_{\Q_p}\int_{L_p}\psi_{p}(bQ(x))dx\psi_{p}(-mb)db.
\end{equation}
Then one has $$W_{p}( s_{0}, m)=\frac{W_{m, p}(e, s_{0}, \cha (L_p))}{\gamma(V_p)\mid \det G  \mid_{p} ^{\frac{1}{2}}}.$$

\begin{lemma}\label{lemnorm}
Let $V=(B^0(2), Q)$ and the lattice $\mathcal{L}=\Z i+\Z j+\Z k$, which are defined in the Section \ref{foursquare}. One has
\begin{equation}
\gamma(V_p)\mid \det G  \mid_{p} ^{\frac{1}{2}}=
\begin{cases}
1 & \ff  p~is ~odd,\\
-\zeta_{8}^{-1}\frac{1}{2\sqrt{2}} & \ff  p=2.
\end{cases}
\end{equation}
\end{lemma}
\begin{proof}
 When $p$ is odd, $\gamma(V_p)=\mid \det G  \mid_{p} ^{\frac{1}{2}}=1$. When $p=2$, $\gamma(V_p)=-\zeta_{8}^{-1}$ and $\mid \det G  \mid_{p} ^{\frac{1}{2}}=\frac{1}{2\sqrt{2}}$.
 Then one obtains the result.  
\end{proof}
 
For the convenience, we rewrite the Theorem \ref{localequality}.
\begin{theorem} 
 1)
  Let $\mathcal{L}=\Z i+\Z j+\Z k$ be the lattice in quadratic space $V=(B^0(2), Q)$, one has 
\begin{equation}S_{p}(m)=W_{ p}(\frac{1}{2}, m).
\end{equation}
2) Let $L=\Z+\Z i+\Z j+\Z k$ be the lattice in $V=(B(2), Q)$, one has
\begin{equation}S_{p}(m)=W_{ p}(1, m).
\end{equation}
where $W_{ p}(\frac{1}{2}, m)$ and $W_{ p}(1, m)$ are normalized local Whittaker function  defined by equation (\ref{norwhitt}).
\end{theorem}
\begin{proof}
We only prove the case 1). 
Assume that $-4m=dc^2$ with fundamental discriminant $d$, and let $\chi_{d}$ be the associated Dirichlet character. Set $r_p=\ord_p(m)$ and $t_p=\ord_p(c)$.

When $p$ is odd, 
\begin{eqnarray}
&&S_{p}(m)= \sum_{r=0}^{r_p+1}A_{p^r}(m)\\
&=&1+\sum_{r=1, ~even}^{r_p}(p-1)p^{-\frac{r}{2}-1}+A_{p^{r_p+1}}\nonumber
\\
&=&\begin{cases}
 1+p^{-1}-p^{-t_p-1}+p^{-t_p-1}\chi_{d}(p) &\ff r_p~is~even,\\
 1+p^{-1}-p^{-t_p-1}-p^{-t_p-2} & \ff r_p~is~odd.
 \end{cases}\nonumber
\end{eqnarray}
Notice that $r_p=2t_p$($r_p=2t_p+1$) when $r_p$ is even(odd).
By equation (\ref{threeodd}) and Lemma \ref{lemnorm}, one obtains 
$$S_{p}(m)=W_{m, p}(e, \frac{1}{2}, \cha(\mathcal{L}_{p}))=W_{p}( \frac{1}{2}, m).$$
When $p=2$,
\begin{eqnarray}
&&S_{2}(m)= \sum_{r=0}^{r_2+3}A_{2^r}(m)\\
&=&1+\sum_{r=1, ~even}^{r_2+2}2^{-\frac{r-1}{2}} \cos\frac{2^{-r+3}m -3}{4}\pi+\sum_{r=1, ~odd}^{r_2+3}A_{2^{r}}\nonumber
\\
&=&\begin{cases}
 0 &\ff m \equiv 7(\rm mod ~8),\\
 1 &\ff m \equiv 3(\rm mod~ 8),\\
 \frac{3}{2} &\ff m \equiv 1, 2, 5, 6(\rm mod ~8),\\
2  S_{2}(\frac{m}{4}) &\ff 4 \mid m.\\
 \end{cases}\nonumber
\end{eqnarray}
One could check that
\begin{equation}
S_{2}(m)=\frac{3}{2^{t_2+1}}\frac{1-\chi_{d}(2)}{1-\frac{\chi_{d}(2)}{2}}.
\end{equation}
By equation (\ref{threetwo}), we know that 
$$S_{2}(m)=\frac{W_{m, 2}(e, \frac{1}{2}, \cha (\mathcal{L}_2))}{-\zeta_{8}^{-1}2\sqrt{2}}=W_{2}( \frac{1}{2}, m).$$ Thus we finish the proof of 1) and leave case 2) to the readers.
\end{proof}

The following result is a directly application of the above theorem.
\begin{theorem} \cite{Ba}
Let the notations be as above, one has
\begin{equation}
r_{s}(m)=\rho_{s}(m), ~s=3, 4.
\end{equation}
\end{theorem}
\begin{proof}
We just prove the case $s=3$. The case $s=4$ is similarly.
From the Theorem \ref{threeth}, one knows that 
\begin{eqnarray}\label{hardyequ}
&&r_{3}(m)\nonumber\\
&=&v^{-\frac{3}{4}}q^{-m}\prod_{p< \infty}W_{m, p}(e, \frac{1}{2}, \Phi_{p})\times W_{m, \infty}(g_{\tau}, \frac{1}{2}, \Phi_{\infty})\nonumber\\
&=&v^{-\frac{3}{4}}q^{-m}\prod_{p< \infty} \gamma(V_p)\mid \det G  \mid_{p} ^{\frac{1}{2}} W_{ p}(\frac{1}{2}, m)\times W_{m, \infty}(g_{\tau}, \frac{1}{2}, \Phi_{\infty})\nonumber\\
&=&-\zeta_8^{-1}\frac{1}{2\sqrt{2}}v^{-\frac{3}{4}}q^{-m}\prod_{p}S_{p}(m) 
W_{m, \infty}(g_{\tau}, \frac{1}{2}, \Phi_{\infty}) \nonumber
\end{eqnarray}
By equation (\ref{threeinfinite}),  we obtain the expected result 
$$r_{3}(m)=2\sqrt{m}\pi\prod_{p}S_{p}(m)=\varrho_{3}(m).$$
\end{proof}

\section{acknowledgment}
Thanks to ..., we write it later.


\end{document}